\documentclass[preprint,1p]{elsarticle}
\usepackage{amsmath,amssymb,amsthm,mathrsfs,calc,color}
\usepackage{graphicx}
\usepackage[british]{babel}

\newtheorem {theorem}{Theorem}[section]
\newtheorem {proposition}[theorem]{Proposition}
\newtheorem {lemma}[theorem]{Lemma}
\newtheorem {corollary}[theorem]{Corollary}

{\theoremstyle{definition}

}
{\theoremstyle{theorem}

}

\def\ba{\begin{array}}
\def\ea{\end{array}}
\def\bea{\begin{eqnarray} \label}
\def\eea{\end{eqnarray}}
\def\be{\begin{equation} \label}
\def\ee{\end{equation}}
\def\bit{\begin{itemize}}
\def\eit{\end{itemize}}
\def\ben{\begin{enumerate}}
\def\een{\end{enumerate}}

\def\EE{\mathbb{E}}

\def\LL{\mathbb{L}}
\def\NN{\mathbb{N}}
\def\PP{\mathbb{P}}
\def\QQ{\mathbb{Q}}
\def\RR{\mathbb{R}}

\def\ZZ{\mathbb{Z}}

\def\d{\delta}

\def\r{\varrho}

\def\G{\Gamma}

\def\L{\Lambda}

\def\cE{\mathcal{E}}
\def\cF{\mathcal{F}}

\def\cH{\mathcal{H}}

\def\cN{\mathcal{N}}
\def\cP{\mathscr{P}}

\def\cS{\mathcal{S}}
\def\cT{\mathcal{T}}

\def\dint{\textup{d}}

\allowdisplaybreaks

\begin{document}

\title{A Mecke-type formula and Markov properties for\\ STIT tessellation processes}

\author{Werner Nagel\corref{cor1}\fnref{fn1}}
\ead{werner.nagel@uni-jena.de}
\address[fn1]{Institut f\"ur Mathematik, Friedrich-Schiller-Universit\"at Jena, Germany}
\cortext[cor1]{Corresponding author}

\author{Linh Ngoc Nguyen\fnref{fn1}}
\ead{rubybell412@gmail.com}

\author{Christoph Th\"ale\fnref{fn2}}
\ead{christoph.thaele@rub.de}
\address[fn2]{Fakult\"at f\"ur Mathematik, Ruhr-Universit\"at Bochum, Germany}

\author{Viola Wei\ss \fnref{fn3}}
\ead{viola.weiss@eah-jena.de}
\address[fn3]{FB Grundlagenwissenschaften, Ernst-Abbe-Hochschule Jena, Germany}

\begin{abstract}
An analogue of the classical Mecke formula for Poisson point processes is proved for the class of space-time STIT tessellation processes. From this key identity the Markov property of a class of associated random processes is derived. This in turn is used to determine the distribution of the number of internal vertices of the typical maximal tessellation segment.
\end{abstract}


\maketitle

\section{Introduction}

The last decades in stochastic geometry have seen a growing interest in models that deal with random geometric objects evolving in time. As examples we mention random sequential packings \cite{PenroseYukich2002,SchreiberPenroseYukich}, spatial birth and growth models like the Johnson-Mehl growth process \cite{BaryYukich,PenroseYukich2002}, the construction of polygonal Markov random fields \cite{Schreiber05,Schreiber08,Schreiber10}, falling/dead leaf models \cite{Bordenave06,cowan94,Galerne12}, on-line geometric random graphs such as the on-line nearest neighbour graph \cite{PenroseWade,Wade09} or the geometric preferential attachment graph \cite{JacobMoeters,Jordan10,Jordan15}. A particularly attractive class of models studied in stochastic geometry is that of random tessellations. Also within this class, space-time models have found considerable interest. In the present paper we investigate the class of STIT tessellations, which arise as outcomes of a process of consecutive cell divisions. They have been invented in \cite{nagel/weiss05} and since their introduction they have stimulated lots of research, cf.\  \cite{DT,l-r10,Mart14,martnag,MartNagel15,MartNagel16,m/n/w11,nagel/biehler15,nagel/weiss08,STPlane,STSpa,STSecond,STBernoulli,st,ThWeiss10,ThWeiss13}.

The STIT tessellation process itself is a Markov process on the space of tessellations. However, there are several interesting situations in which the problem arises whether or not some classes of associated processes also possess the Markov property. For example, we look at the random process induced by the functional of total surface area within in a bounded window. It becomes clear that this process does not inherit the Markov property from the STIT tessellation process, because this functional does not contain enough information about the tessellation. For this reason, it is an interesting task and one of the main purposes of this paper to extract a class of processes that do inherit the Markov property. Such a Markov property will turn out to be a crucial device in further applications. In the present paper we will deal with the distribution of the number of internal vertices of a typical (and possibly weighted) maximal segment of a STIT tessellation. In particular, we will derive the exact distribution of this random variable and study its moment properties.

One of the crucial steps on our way is to prove a Mecke-type formula for STIT tessellations. Such an identity is well known for Poisson point processes $\Gamma$. It says that the expectation of random variables of the form $\sum_{x\in\Gamma}f(x,\Gamma)$ can be expressed as an expectation of the integral with respect to the intensity measure of $\Gamma$ of the function $f(x,\Gamma+\d_x)$, where the unit-mass Dirac measure $\d_x$ concentrated at $x$ has been added to $\Gamma$, see Chapter 4 in \cite{LastPenroseBook} and also \eqref{eq:SlivnyakMecke} below. Our key technical result, Theorem \ref{theor:slivnyakgeneratorSTIT}, provides a formula for STIT tessellations that is similar in its structure to Mecke's equation for Poisson point processes.

\medskip

The present paper is structured as follows. In Section \ref{sec:Preliminaries} we set up the notation, formally introduce STIT tessellation processes by their Markovian description and collect those properties that are needed. The new results are presented in Section \ref{sec:Results}. In particular, we present there our Mecke-type formula, the Markov properties described above as well as an application to maximal segments. The final Section \ref{sec:proof} contains all the proofs. They are technically quite involved and they crucially depend on the 'direct' global construction of STIT tessellations. For this reason, we have also included a formal description of this construction together with its key properties.  Our proofs contain a lot of formalism and intricate calculations. But such a tedious work for STIT tessellations is like that one for other models: To get deeper results and strict proofs, a high degree of formalisation is necessary. This can also be seen, for example, in  \cite{muche95,muche96,muche98} for Voronoi tessellations.

\section{Preliminaries}\label{sec:Preliminaries}

\subsection{Notation}

Let $\RR^d$ be the Euclidean space of dimension $d\in\{1,2,\ldots\}$. A polytope $p\subset\RR^d$ is the convex hull of a finite point set (containing at least two elements) and the dimension of $p$ is defined as the dimension of its affine hull. The set of all polytopes of dimension $k$ is denoted by $\cP_k$. Moreover, we shall write $\cP_k^0$ for the set of all $k$-dimensional polytopes with their circumcenter at the origin. The spaces $\cP_k$ and $\cP_k^0$ are supplied with the Borel $\sigma$-fields $\mathfrak{B}(\cP_k)$ and $\mathfrak{B}(\cP_k^0)$ induced by the  Hausdorff-distance, respectively. In our context it is convenient to speak of the elements of $\cP_d$ as cells and to denote them by the letter $z$ (for the German word `Zelle'). The interior and the boundary of a set $B\subset\RR^d$ are denoted by $\stackrel{\circ}{B}$ and $ \partial B$, respectively. Moreover, we write $\#(\,\cdot\,)$ for the cardinality of the argument set and ${\bf 1}\{\ldots\}$ for an indicator function, which takes the value $1$ if the condition in brackets is satisfied and $0$ otherwise.

The Lebesgue measure on $\RR^d$ is denoted by $\ell_d$.  Moreover, for $d=1$, we write $\ell_+$ and $\ell_-$ for the Lebesgue measure on the positive and the negative real half-axis $(0,\infty)$ and $(-\infty,0)$, respectively. To simplify the notation in integrals, we  write $\dint s$ instead of $\ell_+ (\dint s)$.

If $\cE$ is a topological space, we denote by $\mathfrak{B}(\cE)$ the Borel $\sigma$-field on $\cE$. If $X$ is a random element taking values in the measurable space $[\cE,\mathfrak{B}(\cE)]$ we shall write $\PP_X$ for its distribution, that is, the image of the probability measure of some underlying probability space under $X$. By $X\stackrel{D}{=}Y$ we shall indicate that the $\cE$-valued random elements $X$ and $Y$ have the same distribution, that is $\PP_X=\PP_Y$.

\subsection{The hyperplane measure $\L$}

Let $\cH$ denote the space of all hyperplanes in $\RR^d$ and $\cH_0$ be the subset of hyperplanes containing the origin. Both spaces are supplied with the usual topology of closed convergence (also called Fell topology, see \cite[Chapter A.2]{kal}, \cite[Chapter A.3]{LastPenroseBook} or \cite{schn/weil}) and thus they carry Borel $\sigma$-fields $\mathfrak{B}(\cH_0)$ and $\mathfrak{B}(\cH)$, respectively.

For $h\in \cH$ we shall write $h_0\in \cH_0$ for the parallel linear subspace and $h_0^\perp$ for the one-dimensional subspace orthogonal to it. The two closed half-spaces generated by a hyperplane $h\in\cH\setminus\cH_0$ are denoted by $h^+$ and $h^-$, respectively, where we use the convention that $h^-$ is the half-space that contains the origin. For a Borel set $B\subset \RR^d$ we define
$$
[B] := \{h\in\cH:h\cap B\neq\emptyset\}\,.
$$
This implies $[B]\in\mathfrak{B}(\cH)$ and that $\cH_0=[\{ 0\} ]$.

Let $\QQ$ be a probability measure on $\cH_0$ and  $\ell_{h_0^\perp}$  the Lebesgue measure on the subspace $h_0^\perp$. The translation invariant measure $\L$ is defined by the relation
\begin{equation}\label{eq:decomplambda}
\int_{\cH} g(h)\,\L(\dint h) = \int_{\cH_0}\int_{h_0^\perp} g(h_0+z)\,\ell_{h_0^\perp}(\dint z)\QQ(\dint h_0)
\end{equation}
for all non-negative measurable functions $g:\cH\to\RR$. Throughout this paper we will assume that $\Lambda$ is such that there is no line in $\RR^d$ with the property that all the hyperplanes in the support of $\Lambda$ are parallel to it. This ensures that, with probability one, all cells of the STIT tessellations considered below are bounded.

Because $\Lambda ([q])\in(0,\infty)$ for all polytopes  $q\in\cP_k$ with $k\in\{1,\ldots,d\}$, we can define the probability measure $\Lambda_q$ on $[\cH , {\mathfrak B}(\cH)]$ by
\begin{equation}\label{eq:problambda}
\Lambda_q(B) = \frac{\Lambda (B \cap [q])}{\Lambda ([q])}\,,\qquad B\in\mathfrak{B}(\cH)\,.
\end{equation}

\subsection{Tessellations}

By definition, a tessellation $y$ of $\RR^d$ is  a countable subset of $\cP_d$ satisfying the following three properties:
\begin{itemize}
\item[(i)] $\stackrel{\circ}{z_1}\cap\stackrel{\circ}{z_2}=\emptyset$ for all $z_1,z_2\in y, \ z_1 \not = z_2$,
\item[(ii)] $\bigcup_{z\in y}z=\RR^d$,
\item[(iii)] $\# \{z\in y:z\cap C\}<\infty$ for all compact $C\subset\RR^d$.
\end{itemize}
A `local' tessellation $y$ of a polytope $W\in\cP_d$ is a finite collection of polytopes contained in $W$ that have disjoint interiors and cover $W$. The set of tessellations of $\RR^d$ is denoted by $\cT$ and we write $\cT_W$ for the set of tessellations of a polytope $W\in\cP_d$. A natural way a local tessellation arises is via restriction to $W$ of a global tessellation. Formally, if $y\in\cT$ and $W\in\cP_d$, we define such a restriction by $y\wedge W:=\{z\cap W:z\in y,\dim(z\cap W)=d\}\in\cT_W$, where $\dim(z\cap W)$ denotes the dimension of the polytope $z\cap W$.

Next, we supply $\cT$ and $\cT_W$ with suitable $\sigma$-fields. For this, we recall that the vague topology on $\cT$ is the topology on $\cT$ induced by functions of the form
$$
\cT\to\RR,  \quad y\mapsto\sum_{z\in y}g(z)\,,
$$
where $g:\cP_d\to\RR$ is any non-negative measurable and bounded function that has compact support, see \cite[Theorem A2.3]{kal}. Now, we let $\mathfrak{B}(\cT)$ be the Borel $\sigma$-field generated by the vague topology on $\cT$.  For $W\in\cP_d$ and $\cT_W$ let $\mathfrak{B}(\cT_W)$ be defined analogously. One can check that the restriction map $y\mapsto y\wedge W$ then becomes measurable.

\subsection{The local STIT tessellation process and maximal polytopes}\label{sect:constrbound}

The random STIT tessellation process is denoted by $\underline{Y}$ and its state at time $t$ by  $Y_t$. Informally, the dynamics of the continuous time random local STIT tessellation process $\underline{Y}\wedge W=(Y_t\wedge W)_{t > 0}$ can be described as follows. At time zero, the initial tessellation $W$ receives an exponentially distributed random lifetime with parameter $\L([W])$. When the lifetime of $W$ is running out, a random hyperplane $h\in[W]$ is selected according to the probability distribution $\L_W$, given in (\ref{eq:problambda}),  and splits $W$ into the two sub-polytopes $W\cap h^+$ and $W\cap h^-$. These two polytopes now receive conditionally  independent (given $h$) exponentially distributed random lifetimes with parameters $\L([W\cap h^+])$ and $\L([W\cap h^-])$, respectively, and now evolve independently according to the same rules, i.e., $W\cap h^+$ is divided by a random hyperplane with law $\L_{W\cap h^+}$, and $W\cap h^-$ is divided by a random hyperplane with law $\L_{W\cap h^-}$, and so on. The cells in $W$ that arise at time $t$ form the local STIT tessellation $Y_t\wedge W$. 

To describe the construction formally, let $W\in\cP_d$ be a polytope and $\L$ be a hyperplane measure defined in \eqref{eq:decomplambda}. For a tessellation $y\in\cT_W$, a cell $z\in y$ and a hyperplane $h\in[\stackrel{\circ}{z}]\setminus \cH_0$ we define the splitting operation $\oslash_{z,h}:\cT_W \to \cT_W$ by
$$
\oslash_{z,h}(y) := (y\setminus \{ z \} )\cup\{z\cap h^+,z\cap h^-\}\,.
$$
In other words, $\oslash_{z,h}(y)$ is the tessellation that arises from $y$ by splitting the cell $z$ by means of the hyperplane $h$. The splitting operation is measurable and extends to global tessellations $y\in\cT$ as well.

By the local STIT tessellation process $(Y_t\wedge W)_{t\geq 0}$ in $W$ driven by the hyperplane measure $\L$ we understand the continuous time pure jump Markov process on $\cT_W$ with initial tessellation $Y_0\wedge W=W$ and generator
$$
\LL g(y) := \sum_{z\in y}\int_{[z]} [g(\oslash_{z,h}(y))-g(y)]\,\Lambda(\dint h)\,,\qquad y\in\cT_W\, ,
$$
for all non-negative measurable $g:\cT_W \to \RR$.

\subsection{The global STIT tessellation process and its maximal polytope process}\label{sec:Consistency}

So far we have described the STIT tessellation process locally within polytopes $W\in\cP_d$. However, there exists also a `global' construction of a STIT tessellation process in $\RR^d$. Since this construction is rather involved and is needed only as a technical device in our proofs, we decided to postpone its description to Section \ref{sec:proof} below. For the moment it is sufficient to confirm that such a process exists. For this, we recall from \cite{nagel/weiss05} the following consistency property. Given two polytopes $W,W'\in\cP_d$ with $W'\subset W$, the law of $(Y_t\wedge W)\wedge W'$ coincides with that of $Y_t\wedge W'$, where, recall, for a tessellation $y\in\cT$ and a polytope $W\in\cP_d$, $y\wedge W$ stands for the restriction of $y$ to $W$. For all $t>0$, this consistency property  together with the consistency theorem for random closed sets \cite[Theorem 2.3.1]{schn/weil} yield  the existence of a random tessellation $Y_t$  with the property that its restriction to any $W\in\cP_d$ has the same distribution as the previously constructed local STIT tessellation $Y_t\wedge W$. The translation invariance of the hyperplane measure $\L$ also ensures that the law of $Y_t$ is invariant under  translations. One can also show that consistency extends to the finite-dimensional distributions of the processes $\underline{Y}\wedge W =(Y_t \wedge W)_{t>0}$. This way, the classical Kolmogorov extension theorem ensures the existence of a global STIT tessellation process $\underline{Y}=(Y_t)_{t>0}$ with the appropriate finite-dimensional distributions.

We will use the notation $\underline{Y}=(Y_t)_{t>0}$ for the random STIT tessellation process, $Y_t$ for its state at time $t>0$.  Respective realizations are denoted by $\underline{y}$ and $y_t$.
The distribution of $(Y_t)_{t>0}$ is written $\PP_{\underline Y}$, and correspondingly the distributions of the other random objects. Furthermore, for a tessellation $y_t$ denote $\partial y_t =\bigcup_{z\in y_t} \partial z$.

As  described above, any extant cell $z$ in a STIT tessellation has a random lifetime, and at the end of its lifetime, at time $s$ say, it is divided by a hyperplane $h$. Then we call $(p,s)\in \cP_{d-1} \times (0,\infty )$ with  $p=z\cap h$ a maximal $(d-1)$-polytope, marked with its birth time $s$.

We emphasize that after its birth, a maximal $(d-1)$-polytope can be intersected by other maximal $(d-1)$-polytopes and thus be subdivided further, independently in both of the half-spaces generated by $h$, i.e., in the two cells adjacent to the maximal polytope. But regardless of such events, it will be referred to as a birth time marked maximal polytope, at all times after its birth.

For any $t>0$ we denote by $M_t=M_t(Y_t)= \sum_{(p,s)\in M, s < t} \delta_{(p,s)}$ the point process of all birth time marked  maximal $(d-1)$-polytopes of the global STIT tessellation $Y_t$. Thus $M_t$ is a point process on the product space $\cP_{d-1}\times(0,\infty)$, i.e., it is a random variable with values in
$\cN(\cP_{d-1}\times(0,\infty))$, the set of locally finite counting measures on $\cP_{d-1}\times(0,\infty)$,
supplied with the Borel $\sigma$-field $\mathfrak{B}(\cN(\cP_{d-1}\times(0,\infty)))$ induced by the vague topology.
As usual, we write $(p,s)\in M_t$ if $M_t (\{ (p,s)\} )>0$.

By $M=M((Y_t)_{t>0})$ we denote the random point process of birth time marked  maximal $(d-1)$-polytopes pertaining to the STIT process $\underline{Y}=(Y_t)_{t>0}$.
Also $M$ is a point process on the state space $\cP_{d-1}\times(0,\infty)$.

We emphasize that, given a realization $m$ of a birth time marked maximal polytope process, one can uniquely reconstruct the trajectory $\underline{y}(m)=(y(m_t))_{t>0}$ of a STIT tessellation process that has $m$ as the realization of the pertaining  maximal $(d-1)$-polytope process.

\section{Results}\label{sec:Results}

\subsection{A Mecke-type formula for STIT tessellations}

For a realization $m$ of the birth time marked process $M$ of  maximal $(d-1)$-polytopes we use the notation $m_{(+t)}:= \{ (p,s+t):\, (p,s)\in m  \} $, to express a time shift by $t$ of all the birth times. Furthermore, for $(p,s)\in m$ we denote by $z(p,s)\in \underline{y}(m)$ the uniquely determined cell in the trajectory $\underline{y}(m)$ that is divided at time $s$ by the maximal polytope $p$. Finally, for a cell $z\in\cP_d$ denote by 
$$
m \wedge z =\{ (p\cap z , s): \, (p,s)\in m,\ p \cap {\stackrel{\circ}{z}}\not= \emptyset \} ,
$$
the restriction of $m$ to $z$.

We are now prepared to present the first main result of this paper, that may be regarded as a Mecke-type formula for STIT tessellations as discussed in some detail after its statement. We postpone the proof to Section \ref{sec:proof} below.

\begin{theorem}\label{theor:slivnyakgeneratorSTIT}
Let $M$ be the process of birth time marked  maximal $(d-1)$-polytopes of a (global) STIT tessellation process $(Y_t)_{t>0}$ driven by a hyperplane measure $\Lambda$. Let $\PP_M$ be the distribution of $M$ and $\PP_{Y_s}$ be that of $Y_s$ at time $s>0$.  Then
\begin{align}\label{eq:slivnyakgeneratorSTITspec}
&\int \sum_{(p,s)\in m} g\big(m \wedge z(p,s) , z(p,s) , p, s\big)\,\PP_M(\dint  m) \notag\\
&=\int \int   \sum_{z\in  y_s}  \int \Bigg[ \int \int 
g\big((z\cap h)\cup(  m^{(1)}_{(+s)} \wedge (z\cap h^+)) \cup (  m^{(2)}_{(+s)} \wedge (z\cap h^-) ),    z,z\cap h, s\big) \notag\\
&\hspace{4cm} \PP_M(\dint   m^{(1)})\,  \PP_M(\dint   m^{(2)}) \Bigg]\, \Lambda_z (\dint h) \, \Lambda ([z])\,  \PP_{Y_s}(\dint { y_s}) \, \dint s
\end{align}
for all non-negative measurable functions $g:\cN(\cP_{d-1}\times(0,\infty))\times\cP_d\times\cP_{d-1}\times(0,\infty)\to\RR$.
\end{theorem}

 Theorem \ref{theor:slivnyakgeneratorSTIT} shares some similarities with the Mecke formula for Poisson point processes. To re-phrase the latter, let $\G$ be a Poisson point process with $\sigma$-finite intensity measure $\mu$ and distribution $\PP_\G$ on a measurable space $[\cE,\mathfrak{B}(\cE)]$. Then 
\begin{equation}\label{eq:SlivnyakMecke}
\int\sum_{x\in\gamma}g(x,\gamma)\,\PP_\Gamma(\dint\gamma) = \int\int g(x,\gamma+\d_x)\,\mu(\dint x)\, \PP_\Gamma(\dint\gamma)
\end{equation}
for all non-negative measurable functions $g:\cE\times\cN(\cE)\to\RR$, see Chapter 4 in \cite{LastPenroseBook}.
Obviously, $M$ is not a Poisson point process, but formally, the left-hand side of \eqref{eq:slivnyakgeneratorSTITspec} has the same structure as the left-hand side of the Mecke formula for Poisson point processes. Moreover, on the right-hand side of \eqref{eq:slivnyakgeneratorSTITspec} we see that an additional hyperplane $h$ is introduced at time $s$ (applying the intensity measure $\Lambda_z (\dint h) \, \Lambda ([z]) \, \dint s$), which is similar to the right-hand side of the Mecke formula. The main differences are that in \eqref{eq:slivnyakgeneratorSTITspec} for {\em each} cell $z\in y_s$ a hyperplane is added and moreover, after the division of a cell $z$ by a hyperplane $h$, realizations (denoted $m^{(1)}$ and $m^{(2)}$) of independent copies of $M$ are needed to continue the process in time. 

It is also interesting to compare our Theorem \ref{theor:slivnyakgeneratorSTIT} with Theorem 3.1 in \cite{GST}, where the class of so-called branching random tessellation has been investigated. These tessellation processes constitute a far reaching generalization of the concept of STIT tessellation processes and allow, in particular, for the interaction of cells during the random cell division process as well as for marks (colours) attached to the cells that are also allowed to influence the cell splitting mechanism. Specialized to our context, this result says that for any fixed $t>0$,
\begin{equation}\label{eq:GST}
\begin{split}
&\int \sum_{(p,s)\in m_t}g\big( (y(m_{r}))_{r< s}    ,z(p,s),p,s\big)\,{\bf 1}\{s\leq t\}\,\PP_{M_t}(\dint m_t)\\
&\qquad = \int\int\sum_{z\in y_s}\Big[\int g\big((y_r)_{r\leq s},z,z\cap h,s\big)\,\L_z(\dint h)\,{\bf 1}\{s\leq t\}\Big]\L([z]) \PP_{Y_s}(\dint { y_s}) \, \dint s
\end{split}
\end{equation}
for all non-negative measurable functions $g:\{ \cT ^{(0,s]},\, 0<s<t\} \times\cP_d\times\cP_{d-1}\times(0,t)\to\RR$. 

Here, $\cT ^{(0,s]}$ stands for the class of all measurable mappings from $(0,s]$ to $\cT$ which contain the realisations of a STIT tessellation process on the time interval $(0,s]$. We notice that relation \eqref{eq:GST} is confined to a finite time horizon for technical reasons. Another significant difference is that in \eqref{eq:GST} the functions $g$ are allowed to depend on the evolution that took place in the past of a given time $s$ only (this is reflected by the appearance of $(y_r)_{r\leq s}$). In contrast, the function $g$ in Theorem \ref{theor:slivnyakgeneratorSTIT} above can depend on a potentially infinite time horizon, including the evolution \textit{after} the birth of a particular maximal polytope. On the other hand, relation \eqref{eq:GST} allows for functions that do not only depend on the tessellation within the cell in which a maximal polytope is born, but also on its surrounding (and the colors attached to the cells within this surrounding).

\subsection{Application to maximal polytopes}\label{sec:appl}

We consider the $k$-dimensional faces of maximal $(d-1)$-polytopes, and we refer to them as maximal $k$-polytopes, $k=0,...,d-2$. They appear as the intersection of a sequence of $d-k$ maximal polytopes of dimension $d-1$. It is important to note that for dimensions $d \geq 3$ not all  intersections of $d-k$ maximal polytopes (even if the intersection has dimension $k$) are faces of maximal polytopes. To see this, consider e.g. three maximal $(d-1)$-polytopes $p_1,p_2,p_3$ such that $p_1\cap p_2\cap p_3 \not= \emptyset$, and $p_2,p_3$ are located in different half-spaces generated by the hyperplane containing $p_1$. In view of this, the polytopes which generate a  maximal $k$-polytope have to fulfill additional conditions, which will be formalized in the proof of Proposition \ref{lem:kdaxpolintersectSTITlang}. 

For a realization $m$ of the point process $M$ and for $k=0,...,d-1$,  let
$$
((p_1,s_1),\ldots , (p_{d-k},s_{d-k}))\in m^{d-k}
$$ 
denote a tuple of maximal polytopes together with  their birth times. 
We denote such a tuple by $({\bf p,s},k)\in {m}^{d-k}$  if and only if $s_1<\ldots <s_{d-k}$ and $\overline{\bf p}=\bigcap_{i=1}^{d-k} p_i$ is a maximal $k$-polytope of the STIT tessellation process. If $k<d-1$ this is a $k$-dimensional face of a maximal $(d-1)$-polytope. 
 In this case we call 
$(\overline{\bf p},{\bf s} )=\left( \bigcap_{i=1}^{d-k} p_i ,{\bf s}\right)$ a maximal $k$-polytope of the STIT tessellation process, marked with its birth time tuple ${\bf s}$. Accordingly, we denote ${\bf h}=(h_1,\ldots h_{d-k})\in \cH ^{d-k}$ and 
$\overline{\bf h}=\bigcap_{i=1}^{d-k} h_{i}$. If we write $({\bf p,s},k)\in {m}_t^{d-k}$ we mean that  $({\bf p,s},k)\in {m}^{d-k}$ and that $s_{d-k}<t$.

In the following proposition we consider, for a fixed time parameter $t>0$ and a fixed dimension $k\leq d-1$, the set of all birth time marked maximal $k$-polytopes $(\overline{\bf p},{\bf s} )$ and the trace of the STIT tessellation on them, that is, the intersection 
\begin{equation}\label{eq:sqcap}
m_t  \sqcap \ \overline{\bf p}:=(m_t  \setminus \{ (p_1,s_1), \ldots , (p_{d-k},s_{d-k}) \} ) \cap \overline{\bf p}
\end{equation}
of $\overline{\bf p}$ with the other maximal $(d-1)$-polytopes of $m_t$. Note that $m_t  \sqcap \ \overline{\bf p}$ describes the tessellation structure induced by $m_t$ in the (relative) interior of the maximal $k$-polytope $\overline{\bf p}$.

\begin{proposition}\label{lem:kdaxpolintersectSTITlang}
For $t>0$ all non-negative measurable functions $g: \cP_k\times(0,t)^{d-k}\times \mathfrak{B}(\RR^d) \to \RR$, 
\begin{align*}\label{eq:kdmaxpolintersectSTITlang}
& \int \displaystyle {\sum_{({\bf p,s},k)\in {m_t^{d-k}}} }
g\left( \overline{\bf p}, {\bf s}, m_t  \sqcap \ \overline{\bf p} \right) 
\PP_{{M_t}} (\dint  m_t) \\
&=
2^{d-k-1} \int \ldots \int \sum_{z\in  y_{s_{d-k}}} \textstyle{
g\Big(  z\cap \overline{\bf h},{\bf s} , 
 z\cap \overline{\bf h} \cap \left[ \bigcup\limits_{i=1}^{d-k-1} \partial y_{t -s_{i}}^{(i)} \cup \partial y_{t -s_{d-k}}^{+} 
 \cup \partial y_{t -s_{d-k}}^{-} \right] \Big) \  }\\
&\qquad\PP_{\underline Y}^{\otimes (d-k+1)} (\dint ({\underline y^{(1)}}, \ldots ,
{\underline y^{(d-k-1)}},{\underline y^{+}},{\underline y^{-}}))
\ {\bf 1}\left\{  z\cap \overline{\bf h}\not= \emptyset \right\} \, \Lambda^{\otimes (d-k)} (\dint {\bf h}) \PP_{Y_{s_{d-k}}} (\dint {y_{s_{d-k}}}) \\
&\qquad  \cdot{\bf 1} \{ 0< s_1<\ldots < s_{d-k}<t\}\,  \dint s_1 \ldots \dint s_{d-k} \,.
\end{align*}
\end{proposition}

If the function $g$ in the previous lemma depends on the birth time marked maximal $k$-polytope only, then the restriction to a fixed time $t>0$ can be omitted, and the result can be modified as follows.

\begin{corollary}\label{lem:kdaxpolintersectSTIT}
For all non-negative measurable functions $g:\cP_k\times(0,\infty )^{d-k} \to \RR$, 
\begin{eqnarray*}\label{eq:kdmaxpolintersectSTIT}
&& \int \displaystyle {\sum_{({\bf p,s},k)\in {m^{d-k}}} }
g\left( \overline{\bf p}, {\bf s}\right) 
\PP_{{M}} (\dint  m) 
\\ &&\\
&=& 2^{d-k-1} \int \ldots \int \sum_{z\in  y_{s_{d-k}}} \textstyle{
g(  z\cap \overline{\bf h},{\bf s})\  }  
\  {\bf 1}\left\{  z\cap \overline{\bf h}\not= \emptyset \right\} \, \Lambda^{\otimes (d-k)} (\dint {\bf h}) \PP_{Y_{s_{d-k}}} (\dint { y_{s_{d-k}}})
\\ &&  \\
&&  
%\\ && \\
%&& 
\qquad  \qquad \qquad \qquad\cdot{\bf 1} \{ 0< s_1<\ldots < s_{d-k}\}\,  \dint s_1 \ldots \dint s_{d-k} \,.
  \end{eqnarray*}
\end{corollary}

\subsection{Densities and distributions of typical polytopes}
For a fixed $k\in\{0,\ldots,d-1\}$ and a fixed time $t>0$, let us consider the marked  point process $\Phi_t$ of circumcenters of maximal $k$-polytopes of the STIT tessellation $Y_t$, which we mark with the maximal $(d-1)$-polytopes and their birth times and with the  `internal structure' of the  maximal $k$-polytopes induced by $M_t$.

For a polytope $q$ denote by $c(q)$ its circumcenter and define the mapping 
$$
m_t \mapsto \{ (c(\overline{\bf p}), \overline{\bf p}-c(\overline{\bf p}), {\bf s},  m_t  \sqcap \ \overline{\bf p}):\ ({\bf p,s},k)\in {m_t^{d-k}} \}\,.
$$ 
By $\PP_{\Phi _t }$ we denote the image measure of $\PP_{M_t}$ under this mapping. Next, we define
\begin{equation}\label{eq:rkj}
\r_{k,t}^{(j)} :=\frac{1}{\ell_d(B)}\int\sum_{(x,q, {\bf s},T)\in \varphi_t}{\bf 1}\{ x\in B \} V_j(q)\,\PP_{\Phi_t}(\dint \varphi_t)
\end{equation}
for $B\in\mathfrak{B}(\RR^d)$ with $0<\ell_d(B)<\infty$, $t>0$, $k\in\{0,\ldots,d-1\}$, and  $V_j$ is the $j$th intrinsic volume, $j\in\{0,\ldots,k\}$. In particular, $\r_{k,t}^{(0)}$ is the intensity of the point process of circumcenters of maximal $k$-polytopes and in general, $\r_{k,t}^{(j)}$ is the mean cumulative (or total) $V_j$th intrinsic volume of all maximal $k$-polytopes per unit volume.

Campbell's theorem \cite[Proposition 2.7]{LastPenroseBook} implies that the probability measure $\PP_{\Phi_t}$ can be dis\-integrated, that is, there exists a probability measure $\QQ_{(\overline{\bf P}, {\boldsymbol{\beta}},\tau ),t}$ such that the Palm formula
\begin{equation}\label{eq:kdmaxpolPalm}
\begin{split}
&\int \sum_{(x,q,{\bf s}, T)\in \varphi_t} g(x, q, {\bf s},T)\,\PP_{\Phi _t} (\dint \varphi_t) \\
&\qquad =\r_{k,t}^{(0)}\int  \int g(x, q, {\bf s}, T ) \, \QQ_{(\overline{\bf P}, {\boldsymbol{\beta}}, \tau ),t} (\dint (q, {\bf s}, T)) \ell_d (\dint x)
\end{split}
\end{equation}
holds for all non-negative measurable functions $g:\RR^d\times \cP_k\times(0,t)^{d-k}\times \mathfrak{B}(\RR^d) \to \RR$.  A random $k$-dimensional polytope of  $\cP_k$ (endowed with the tuple of its birth times and the internal structure on it)  with distribution $\QQ_{(\overline{\bf P}, {\boldsymbol{\beta}}, \tau ),t}$ is called a typical maximal $k$-polytope of the tessellation $Y_t$.

In what follows, we also consider typical weighted  maximal $k$-polytopes of $Y_t$, with the intrinsic volumes $V_j$, $0\le j \le k$, as weights. Their distribution $\QQ^{(j)}_{(\overline{\bf P}, {\boldsymbol{\beta}},\tau ),t}$ is defined by the weighted Palm formula
\begin{equation}\label{eq:weighting}
\begin{split}
&\left[ \int V_j(q) \,\QQ_{\overline{\bf P},t}(\dint q)\right]^{-1} \int \sum_{(x,q,{\bf s}, T)\in \varphi_t}  V_j(q) g(x, q, {\bf s},T)\,\PP_{\Phi _t} (\dint \varphi_t)\\
&\qquad = \r_{k,t}^{(0)} \int  \int g(x, q, {\bf s}, T )\, \QQ^{(j)}_{(\overline{\bf P}, {\boldsymbol{ \beta}}, \tau),t} (\dint (q, {\bf s},T)) \ell_d (\dint x)\,,
\end{split}
\end{equation}
where $\QQ_{\overline{\bf P},t}$ is the marginal distribution of $\QQ_{(\overline{\bf P}, {\boldsymbol{\beta}},\tau ),t}$ for $\overline{\bf P}$. Note that $ \QQ^{(0)}_{(\overline{\bf P}, {\boldsymbol{ \beta}},\tau ),t}= \QQ_{(\overline{\bf P}, {\boldsymbol{ \beta}},\tau ),t}$. Since $\PP_{\Phi _t }$ is the image measure of $\PP_{M_t}$, the right hand side of \eqref{eq:weighting} can be transformed accordingly. Namely, for $d\geq 2$, $k\in\{0,\ldots,d-1\}$, $g: \cP_k\times(0,t)^{d-k}\times \mathfrak{B}(\RR^d) \rightarrow\RR$ non-negative and measurable, and $t>0$, it holds 
\begin{equation}\label{eq:maxpolbirth}
\begin{split}
& \r_{k,t}^{(j)}\, \int g(q, {\bf s},T) \, \QQ^{(j)}_{(\overline{\bf P}, {\boldsymbol{ \beta}},\tau ),t} (\dint (q, {\bf s},T))\\
&\qquad = \int \sum_{({\bf p,s},k)\in {m_t^{d-k}}}  V_j(\overline{\bf p})\cdot {\bf 1}\{c(\overline{\bf p})\in {[0,1]^d}\}\\
&\qquad\qquad\qquad\qquad\qquad\qquad  \cdot  g(\overline{\bf p}-c(\overline{\bf p}), {\bf s}, (m_t\sqcap \overline{\bf p})-c(\overline{\bf p}))\,\PP_{ M_t} (\dint m_t)\,.
\end{split}
 \end{equation}

Combining  (\ref{eq:rkj}) and (\ref{eq:kdmaxpolPalm}) immediately leads to the identity
\begin{equation}\label{eq:meanvalueproduct}
\r_{k,t}^{(j)} = \r_{k,t}^{(0)} \,  \int  V_j(q) \, \QQ_{\overline{\bf P},t} (\dint q)\,.
\end{equation}
Moreover, using the scaling property \eqref{eq:STITglobalScaling} of STIT tessellations and the homogeneity of  the intrinsic volumes, one easily checks that for $t>0$, $k\in\{0,\ldots,d-1\}$ and $j\in\{0,\ldots,k\}$,
\begin{equation}\label{eq:ScalingRho}
\r_{k,t}^{(j)}=t^{d-j}\r_{k,1}^{(j)} .
\end{equation}

 Proposition \ref{lem:kdaxpolintersectSTITlang}  or Corollary \ref{lem:kdaxpolintersectSTIT} together with an integration with respect to the time coordinates $s_1,\ldots ,s_{d-k-1}$ imply that $ \r_{k,t}^{(j)}$ can be represented as
\begin{equation}\label{eq:rhoformel}
\begin{split}
\r_{k,t}^{(j)} &= 2^{d-k-1} \frac{1}{\ell_d(B)} \int \int \int \sum_{z\in  y_{s_{d-k}}} \textstyle{{\bf 1}\{ c(z\cap \overline{\bf h})\in B \} V_j(z\cap \overline{\bf h})}\\
&\qquad  \PP_{ Y_{s_{d-k}}} (\dint { y_{s_{d-k}}}) \Lambda^{\otimes (d-k)} (\dint {\bf h}) %\\ && \\ &&
{\bf 1} \{ 0<  s_{d-k}< t\}\,  \frac{s_{d-k}^{d-k-1}}{(d-k-1)!} \dint s_{d-k}\,.
\end{split}
\end{equation}
It will be useful to have a more concise representation for $\r_{k,1}^{(j)}$, which in view of \eqref{eq:ScalingRho} is no restriction of generality.

\begin{proposition}\label{lem:rho1}
For all $B\in\mathfrak{B}(\RR^d)$ with $0<\ell_d(B)<\infty$,  $k\in\{0,\ldots,d-1\}$ and $j\in\{0,\ldots,k\}$ one has that
\begin{equation*}\label{eq:rho1}
\begin{split}
 \r_{k,1}^{(j)}&= 2^{d-k-1}  \frac{1}{(d-k-1)!\, (d-j)}\cdot \frac{1}{\ell_d(B)}\\
 &\qquad \cdot\int  \int \sum_{z\in  y_1} \textstyle{
 {\bf 1}\{ c(z\cap \overline{\bf h})\in B \} V_j(z\cap \overline{\bf h})  }  \,   \Lambda^{\otimes (d-k)} (\dint {\bf h}) \PP_{ Y_{1}} (\dint { y_1}) \,.
\end{split}
\end{equation*}
\end{proposition}
 
\subsection{Markov properties of typical maximal polytopes and their birth time distributions}

 We are now going to apply Proposition \ref{lem:kdaxpolintersectSTITlang} to prove the Markov properties for the joint birth time distribution of the typical maximal $k$-polytope. We start by determining the marginal distribution $\QQ^{(j)}_{{\boldsymbol{ \beta}},t}$, that is, the birth time distribution of the typical $V_j$-weighted maximal $k$-polytope of $Y_t$. Our next proposition largely extends and unifies earlier results for the special case $k=d-1$ and $j=0$ in \cite{STBernoulli} and $d=3$, $k=1$ and $j\in\{0,1\}$ in \cite{TWN12}.

\begin{theorem}\label{thm:GeneralBirthTime} 
Let $d\geq 2$, $k\in\{0,\ldots,d-1\}$, $j\in\{0,\ldots,k\}$ and $t>0$. The distribution $\QQ^{(j)}_{{\boldsymbol{ \beta}},t}$ of the birth times ${\boldsymbol{ \beta}}= (\beta_1,\ldots,\beta_{d-k})$ of the typical $V_j$-weighted   maximal $k$-polytope  has the density
$$
(s_1,\ldots,s_{d-k})\mapsto (d-j)(d-k-1)!{s_{d-k}^{k-j}\over t^{d-j}}\,{\bf 1}\{0<s_1<\ldots<s_{d-k}<t\}
$$
with respect to the Lebesgue measure on $\RR^{d-k}$.
\end{theorem}

After this preparation, the following results can be shown by direct computations.

\begin{corollary}\label{cor:marginalconditional}
Let $d\geq 2$, $k\in\{0,\ldots,d-1\}$ and $j\in\{0,\ldots,k\}$. 
\begin{itemize}
\item[(a)] The marginal distribution $\QQ^{(j)}_{\beta_{d-k},t}$ of the last birth time 
of the typical $V_j$-weighted  maximal $k$-polytope  has the density
$$
s_{d-k}\mapsto (d-j){s_{d-k}^{d-j-1}\over t^{d-j}}\,{\bf 1}\{0< s_{d-k}<t\}
$$
with respect to the Lebesgue measure on $\RR$.
\item[(b)] For all $s_{d-k}<t$, the conditional distribution $\QQ^{(j)}_{{( {\beta_1,\ldots , \beta_{d-k-1}} ),t} |\beta_{d-k}=s_{d-k} } $ of the birth times $ (\beta_1,\ldots,\beta_{d-k-1})$ of the typical  $V_j$-weighted  maximal $k$-polytope, given $\beta_{d-k}=s_{d-k}$ has the density
$$
(s_1,\ldots,s_{d-k-1})\mapsto (d-k-1)!\,  s_{d-k}^{-(d-k-1)}\,{\bf 1}\{0<s_1<\ldots<s_{d-k}\}
$$
with respect to the Lebesgue measure on $\RR^{d-k-1}$. In particular, this conditional distribution does not depend on $j$, and it is the uniform distribution on the $(d-k-1)$-simplex $\{(s_1,\ldots,s_{d-k-1})\in\RR^{d-k-1}:0<s_1<\ldots<s_{d-k-1}<s_{d-k}\}$. 
\end{itemize}
\end{corollary}

Furthermore, the marginal distribution $\QQ_{(\overline{\bf P}, \beta_{d-k}),t}$    as well as the conditional distribution $\QQ_{\overline{\bf P},t| \beta_{d-k}=s_{d-k} }$ can  be calculated.

\begin{corollary}\label{lem:maxpolbirth}
Let $d\geq 2$, $k\in\{0,\ldots,d-1\}$, $j\in\{0,\ldots,k\}$, $g:\cP_k\times(0,t) \rightarrow\RR$ be non-negative and measurable and $t>0$. Then,
\begin{align*}
&\int g( q, s_{d-k})\, \QQ^{(j)}_{(\overline{\bf P}, { \beta_{d-k}} ),t} \, (\dint (q,  s_{d-k})) \\
&= 
 2^{d-k-1} \left[ \r_{k,t}^{(j)}\right]^{-1}\, t^{d-j}\\
&\qquad\cdot \int \int \int \sum_{z\in  y_{s_{d-k}}} \textstyle{
V_j(z\cap \overline{\bf h})\cdot {\bf 1}\{c(z\cap \overline{\bf h})\in {[0,1]^d}\})\cdot  g(  (z\cap \overline{\bf h})-c(z\cap \overline{\bf h}),s_{d-k})\  } 
\\ 
& \qquad \qquad \Lambda^{\otimes (d-k)} (\dint {\bf h}) \PP_{ Y_{s_{d-k}}} (\dint { y_{s_{d-k}}}){{s_{d-k}^{-(k-j)}}\over {(d-k-1)! (d-j)}  }     \,   
 \QQ^{(j)}_{\beta_{d-k},t}( \dint s_{d-k})\,.
 \end{align*}
In particular, for almost all $s_{d-k} \in (0,t)$ the conditional distribution $\QQ^{(j)}_{{\overline{\bf P} ,t}| \beta_{d-k}=s_{d-k} }$ is given by
\begin{align*} 
&\int \tilde g( q) \, \QQ^{(j)}_{{\overline{\bf P} ,t}| \beta_{d-k}=s_{d-k} }(\dint q) \\
&= 
 2^{d-k-1} \left[ \r_{k,t}^{(j)}\right]^{-1}{{s_{d-k}^{-(k-j)}}\over {(d-k-1)! (d-j)}  }     \,   
t^{d-j}\, \int \int \int \sum_{z\in  y_{s_{d-k}}} 
V_j(z\cap \overline{\bf h})\\
&\qquad \cdot {\bf 1}\{c(z\cap \overline{\bf h})\in {[0,1]^d}\}\cdot  \tilde g( (z\cap \overline{\bf h})-c(z\cap \overline{\bf h})) \,\Lambda^{\otimes (d-k)} (\dint {\bf h}) \PP_{ Y_{s_{d-k}}} (\dint { y_{s_{d-k}}}) 
 \end{align*}
for all non-negative and measurable $\tilde g :\cP_k \rightarrow\RR$.
\end{corollary}

An application of the result obtained so far yields the following conditional independence property, which can also be interpreted as a Markov property for STIT tessellation processes. To formulate it, let  $\QQ^{(j)}_{{\overline{\bf P} ,t}| \beta_{d-k}=s_{d-k} }$ and $\QQ^{(j)}_{{( {\beta_1,\ldots , \beta_{d-k-1}} ),t} |\beta_{d-k}=s_{d-k} }$ denote  conditional distributions (as indicated by their indexes), pertaining to $\QQ^{(j)}_{(\overline{\bf P}, {\boldsymbol{ \beta}} ),t}$, respectively.

 \begin{theorem}\label{markovbirthCL}
 Let $d\geq 2$, $k\in\{0,\ldots,d-1\}$, $j\in\{0,\ldots,k\}$, $g: \cP_k\times(0,t)^{d-k} \rightarrow\RR$ be non-negative and measurable and $t>0$. Then,
\begin{align*} 
& \int  g(q, {\bf s} ) \, \QQ^{(j)}_{(\overline{\bf P}, {\beta_1,\ldots , \beta_{d-k}}),t}  (\dint (q, {\bf s})) = \int \int \int  g( q, {\bf s} ) \\
&\qquad \QQ^{(j)}_{{\overline{\bf P} ,t}| \beta_{d-k}=s_{d-k} } (\dint q)
  \QQ^{(j)}_{{( {\beta_1,\ldots , \beta_{d-k-1}} ),t} |\beta_{d-k}=s_{d-k} } 
  (\dint (s_1,\ldots , s_{d-k-1}))
  \QQ^{(j)}_{\beta_{d-k}} (\dint s_{d-k} )\,,
\end{align*} 
which is equivalent to the conditional independence of the typical  $V_j$-weighted  maximal $k$-polytope $\overline{\bf P}$ and $({\beta_1,\ldots , \beta_{d-k-1}} )$, given the last birth time $\beta_{d-k}=s_{d-k}$.
\end{theorem}

\subsection{The number of internal vertices on maximal segments}

Now we turn  to an application of the results, considering the maximal segments, i.e., the maximal 1-polytopes of the STIT tessellation $Y_t$ with a driving measure $\L$ as in \eqref{eq:decomplambda}. These segments may have internal vertices (that is, vertices that are located in the relative interior of a segment), which arise already at the time of birth of the segment (when $d\ge 3$) and thereafter subject to further subdivision of adjacent cells. In the planar case, a maximal segment is always born without internal vertices.  The following theorem provides the distribution of the number of internal points of the typical and the typical length weighted maximal segment, respectively. Formally, for $t>0$, $ j=0,1$ and $n=0,1,2,\ldots $ we define for the typical  $V_j$-weighted  maximal $1$-polytope with distribution  $\QQ^{(j)}_{(\overline{\bf P}, {\boldsymbol{ \beta}},\tau ),t}$

\begin{equation}\label{eq:kdmaxpolRefCampInt}
 \mathsf{p}_{1,j}(n) 
:=  \int  \int
{\bf 1}\{ \# T =n \}\,
  \QQ^{(j)}_{(\overline{\bf P}, {\boldsymbol{ \beta}},\tau ),t} (\dint (q, {\bf s},T)) , 
\end{equation}
where the variable $T$ stands for the 'internal structure' as defined in (\ref{eq:sqcap}).

\begin{theorem}\label{thm:GeneralInternalVertices} 
Let $d\geq 2$. For all $n\in\{0,1,2,\ldots\}$ the probabilities $\mathsf{p}_{1,0}(n)$ and $\mathsf{p}_{1,1}(n)$, respectively, are given by
$$
\mathsf{p}_{1,0}(n)=d(d-2)!\int_0^t\int_0^{s_{d-1}}\!\!\!\!\!\!\!\!\ldots\int_0^{s_2}{s_{d-1}^2\over t^d}{(d\cdot t-2s_{d-1}-s_{d-2}-\ldots-s_1)^n\over(d\cdot t-s_{d-1}-s_{d-2}-\ldots-s_1)^{n+1}}\,\dint s_1\ldots\dint s_{d-1}
$$
and
\begin{equation*}
\begin{split}
\mathsf{p}_{1,1}(n)=(n+1)(d-1)!\int_0^t\int_0^{s_{d-1}}\!\!\!\!\!\!\!\!\ldots\int_0^{s_2}{s_{d-1}^2\over t^{d-1}}&{(d\cdot t-2s_{d-1}-s_{d-2}-\ldots-s_1)^n\over(d\cdot t-s_{d-1}-s_{d-2}-\ldots-s_1)^{n+2}}\\ &\dint s_1\ldots\dint s_{d-1}\,.
\end{split}
\end{equation*}
\end{theorem}

Theorem \ref{thm:GeneralInternalVertices} implies in particular that the probabilities $\mathsf{p}_{1,0}(n)$ and $\mathsf{p}_{1,1}(n)$  are independent of the driving measure $\L$. Moreover, the substitution $u_i:=ts_i$, $i\in\{1,\ldots,d-1\}$ shows that these probabilities are also independent of the time parameter $t$. This is consistent with the scaling property (\ref{eq:STITglobalScaling}) of a STIT tessellation, because the number of internal vertices on a maximal 1-polytope does not change when the tessellation is rescaled. Theorem \ref{thm:GeneralInternalVertices} can also be used to compute the moments of the respective distributions. The following identities are readily checked by using Theorem \ref{thm:GeneralInternalVertices}.

\begin{corollary}
Let $d\geq 2$ and $N_d^{(j)}, \ j=0,1,$ be random variables with distributions given by $\mathsf{p}_{1,j}$. Then
\begin{align*}
\EE N_d^{(0)}={1\over 2}{d^2-d+2\over d-1}\qquad\text{and}\qquad \EE N_d^{(1)}=\begin{cases}+\infty & \mbox{ if } d=2\\ {\displaystyle {d^2-2d+4\over d-2}} & \mbox{ if } d\geq 3\,.
\end{cases}
\end{align*}
\end{corollary}

Note that $\EE N_2^{(0)}=\EE N_3^{(0)}$, while $\EE N_d^{(0)}$ is strictly increasing for all $d\geq 3$. In contrast, for the mean number of internal vertices on the typical length-weighted maximal segment we have that $\EE N_3^{(1)}=7$, $\EE N_4^{(1)}=6$, $\EE N_5^{(1)}=6{1\over 3}$ and $\EE N_6^{(1)}=7$ and, considered as a function of $d$, $\EE N_d^{(1)}$ is strictly increasing for $d\geq 5$.

In the planar case $d=2$, as mentioned above, the probabilities $\mathsf{p}_{1,0}(n)$ are known from \cite{m/n/w11,thaele10}, whereas for $d=3$ the formula for $\mathsf{p}_{1,0}(n)$ has been established in \cite{TWN12} by different methods. Our approach in the present paper is more general and allows to deduce the corresponding formula also for the length-weighted maximal segment  as well as to deal with arbitrary space dimensions. As a concrete example, take $d=3$ and consider the length-weighted typical maximal segment. Here, we have
\begin{eqnarray*}
\mathsf{p}_{1,1}(0) &=& 5+18\ln 2-{63\over 4}\ln 3\approx 0.173506\,,\\
\mathsf{p}_{1,1}(1) &=& 28+90\ln 2-{657\over 8}\ln 3\approx 0.159712\,,\qquad {\rm etc.}
\end{eqnarray*}
The mean number of internal vertices is $7$ in this case. The values $\mathsf{p}_{1,1}(n)$ may be determined from the formula in Theorem \ref{thm:GeneralInternalVertices} by straightforward integration.

\section{Proofs} \label{sec:proof}

\subsection{Essential ingredients}

\paragraph{\bf Iteration of tessellations}
The acronym STIT stands for the \underline{st}abi\-lity of distribution under the operation of \underline{it}eration (or nesting) of tessellations. In order to define this operation formally, let $y=\{ z_i:i\in\NN\}$ be a tessellation of $\RR^d$ and
 ${\vec y}=(y^{(i)})_{i\in\NN}$ be  a sequence of tessellations. Then the tessellation $y\boxplus {\vec y}$, referred to as the iteration of $y$ and ${\vec y}$, is specified  by 
\begin{eqnarray}\label{defiteration}
y\boxplus {\vec y} := \{y^{(i)}\wedge z_i : i\in\NN\}\,.
\end{eqnarray}
Thus, for each $i\in\NN$ we restrict the tessellation ${y}^{(i)}$ to the cell $z_i\in y$. This yields a local tessellation of $z_i$ and the union of all these local tessellations clearly forms a tessellation of $\RR^d$. We notice that $\boxplus$ defines a measurable operation from $\cT\times\cT^\NN$ to $\cT$.

Now let $\underline Y=(Y_t)_{t>0}$ be a STIT tessellation process driven by some hyperplane measure $\L$ as in \eqref{eq:decomplambda}, and let
${\vec {\underline Y}}=({\underline Y}^{(i)})_{i\in\NN}$ be a sequence of i.i.d.\ copies of $\underline Y$. For fixed $s>0$, we write ${\vec Y}_s=({Y_s}^{(i)})_{i\in\NN}$. Then 
\begin{equation}\label{iterate}
Y_{t} \stackrel{D}{=} Y_s\boxplus {\vec Y}_{t-s}\qquad\text{for\ all}\qquad 0<s<t\,,
\end{equation}
cf.\ \cite[Lemma 2]{nagel/weiss05}. This implies in particular that $Y_{2t}\stackrel{D}{=} Y_t\boxplus {\vec Y}_t$ for all $t>0$. 
The STIT property means that
\begin{equation*}
\label{stit}
Y_{t} \stackrel{D}{=} 2(Y_t\boxplus {\vec Y}_t)\qquad\text{for\ all}\qquad t>0\,,
\end{equation*}
where multiplication of a tessellation $y\in\cT$ with a factor $2$ stands for the transformation $2y=\{2z:z\in y\}$ and $2z=\{2x:x\in z\}$.

\paragraph{\bf STIT scaling}
The dilation $tY_t$ of $Y_t$ by factor $t$ has the same distribution as $Y_1$, the STIT tessellation with time parameter $1$, that is,
\begin{equation}\label{eq:STITglobalScaling}
 tY_t\overset{D}{=}Y_1\qquad{\rm for\ all}\qquad t>0\,,
\end{equation}
see Lemma 5 in \cite{nagel/weiss05}.

\paragraph{\bf STIT intersections}
The intersection of the STIT tessellation $Y(t)$ with a line $L={\rm span}\,u$, where $u\in\mathcal{S}_+^{d-1}$ (upper unit half-sphere) is a Poisson point process with intensity $t\L([u])$ (here $u$ has to be interpreted as the line segment connecting the origin with $u$), cf.~\cite{nagel/weiss05}.

\subsection{A global construction}\label{globalconstr}

The main technical device in the proof of Theorem \ref{theor:slivnyakgeneratorSTIT} is a \textit{global} construction developed in \cite{m/n/w08a, m/n/w08b,m/n/w11} for the STIT tessellation process with driving measure $\L$. Here we summarize the essential ingredients that are needed for our later purposes.

We start with a Poisson point process $\Pi$ on the measurable space $[\cH \times (0,\infty),\mathfrak{B}(\cH\times(0,\infty))]$ with the intensity measure  $\Lambda \otimes \ell_+$. Now, we define the random process $(\tilde \Pi_t)_{t>0}$ of marked Poisson hyperplane processes,  putting
$$
\tilde \Pi_t :=\{ (h,s)\in \Pi :\, s\leq t \}\,,\qquad t>0\,.
$$
For $(h,s)\in  \Pi$ we interpret $s$ as the birth time of the hyperplane $h$ and write $\beta (h)=s$.

Our assumption on the measure $\L$ ensures that for all $t>0$, the Poisson hyperplane process $\Pi_t=\{h\in\cH:(h,s)\in\tilde\Pi_t\}$ a.s. (almost surely) induces a tessellation of $\RR^d$. We denote this Poisson hyperplane tessellation by $X_t$, and by $\underline{X}=(X_t)_{t>0}$ the corresponding random process. For any $t>0$ there is an a.s. uniquely determined random cell $Z_t^0$ of the Poisson hyperplane tessellation $X_t$ that contains the origin. The random process on $\cP_d$ of these zero cells is denoted by $(Z_t^0)_{t>0}$. Clearly, this process is a pure jump process. Let $(\eta_k)_{k\in\ZZ}$ be the sequence of its jump times with the convention that $\eta_1<1\leq \eta_2$. In \cite[Lemma 4.1]{m/n/w08a} it was shown that
$$
\bigcup_{k\in \ZZ} Z_{\eta_k}^0 =\RR^d\,.
$$
At each jump time $\eta_k$ a cell $\hat Z_k$ is chopped off from the current zero cell. The basic idea is to start immediately within each of these new cells $\hat Z_k$ a local STIT tessellation process as described in Section \ref{sect:constrbound} (with the window $W$ replaced by $\hat Z_k$). This can formally be described as follows.

Let $\Sigma^*$ be a Poisson point process on the measurable space $$\left[ \RR^d\times(-\infty, 0)\times\cN(\cH\times(0,\infty)),\mathfrak{B}(\RR^d\times(- \infty , 0)\times\cN(\cH\times(0,\infty)))\right]$$ 
with intensity measure $\ell_d\times\ell_-\times\PP_\Pi$, where $\PP_\Pi$ is the distribution of the process $\Pi$  defined above. Further, define $\Sigma := \Sigma^* +\delta_{(0,0,\Pi)}$, where we suppose that the point processes $\Sigma^*$ and $\Pi$ are independent. We interpret the points of $\Sigma$ as a collection of random points in $\RR^d$ that are marked with priorities in $(-\infty,0)$ and a birth time marked hyperplane process from $\cN(\cH\times(0,\infty))$, the space of locally finite counting measures on $\cH\times(0,\infty)$.

The points from $\RR^d \times (-\infty ,0]$ are designed to select a hyperplane process which is then used for the division of an extant cell. Namely, given a cell $z\in\cP_d$ we choose the point $(X(z), R(z), \Psi (z)) \in \Sigma$ such that $X(z)\in z$ and $R(z)= \max \{ r\in (-\infty ,0]: (x,r,\psi)\in \Sigma , x\in z \}$.

 In other words, $X(z)$ is the a.s. uniquely determined point in $z$ with the highest priority. Note that after the first division of $z$ this point remains a.s. the same for one of the two daughter cells, while for the other daughter cell a new point is selected. It is clear that if $z$ is the zero cell we always have that $(X(z), R(z), \Psi (z))= (0,0 , \Pi)$. Now, if a cell $z$ is born at time $\beta(z)$ by division of its mother cell or by separating from the current zero cell, and if $(X(z), R(z), \Psi (z))$ is chosen as  described, then the marked hyperplane $(h,s )\in \Psi (z)$ is used to divide $z$ further if and only if $h\in [z]$ and $s  =\min \{ s'>\beta (z) :\, (h',s') \in \Psi (z),\, h'\in [z] \}$. This further division then leads to the birth time marked  maximal $(d-1)$-polytope $(z\cap h , s)$, i.e. $\beta (z\cap h )=s$, and the two new daughter cells have the birth time $s$ as well. 

The construction we have described defines a random process  on the space $\cT$ of tessellations. In fact, it has been shown in \cite{m/n/w08a,m/n/w08b,m/n/w11} that, restricted to a polytope $W$, this process coincides with the local STIT tessellation process in $W$ driven by the hyperplane measure $\L$. As explained in Section \ref{sec:Consistency}, the distribution of this process must then coincide with that one of the global STIT tessellation process $\underline{Y}=(Y_t)_{t>0}$ defined by means of consistency and the Kolmogorov extension theorem. The construction here is an explicit global construction based on the Poisson point process $\Sigma$, and it is the key device in the proof of Theorem \ref{theor:slivnyakgeneratorSTIT}.

\subsection{Proof of Theorem \ref{theor:slivnyakgeneratorSTIT} }

We are now going to give a proof of Theorem \ref{theor:slivnyakgeneratorSTIT}, which makes use of the  global construction outlined in the previous section. We use the same notation as there. Moreover, for a realization $\sigma $ of the Poisson point process $\Sigma$,  let $ m(\sigma )$ be the uniquely determined  realization of the point process of birth time marked maximal $(d-1)$-polytopes. Correspondingly, $\underline{y}( \sigma )$ denotes the realization of the STIT process determined by $\sigma $, and  $y(\sigma)_s$ its state at time $s$. 
By $z \in \underline{y}( \sigma )$ we mean any cell which is extant in some time interval (i.e. between its birth and its division) in the realization $\underline{y}( \sigma )$. Further, for given $(x,r,\psi )\in \sigma$ and $(h,s)\in \psi$, there can be either no or exactly one cell $z \in \underline{y}( \sigma )$ such that $z$ is divided by $h$ at time $s$. We describe this by writing
\begin{align*}
\sum_{z \in \underline{y}( \sigma )}\,
& {\bf 1}\{ x\in z \} \,
{\bf 1}\{ r= \max \{ r'\in (-\infty ,0]: (x',r',\psi' )\in \sigma , x'\in z  \} \}  \\
&\qquad\cdot{\bf 1}\{ h\in [z] \} \, {\bf 1}\{ s  =\min \{ s'>\beta (z) :\, (h',s') \in \psi ,\, h'\in [z] \} \} \,.
\end{align*}
We will use this rather extensive form in the following proof  when we apply the Mecke formula for Poisson point processes.

For a better readability we introduce the following abbreviatory notation for a cell $z \in \underline{y}( \sigma )$ and for a given hyperplane $h$ and a time $s$: 
\begin{eqnarray*}
{\bf 1}(z,\sigma^*,\psi) &=& {\bf 1}\{ x\in z, 0 \not \in z \} \,
{\bf 1}\{ r= \max \{ r'\in (-\infty ,0): (x',r',\psi' )\in \sigma^* , x'\in z  \} \} \cdot \\
&&\cdot{\bf 1}\{ h\in [z] \} \, {\bf 1}\{ s  =\min \{ s'>\beta (z) :\, (h',s') \in \psi ,\, h'\in [z] \} \} ,\\ 
&& {\rm for} \ (x,r,\psi )\in \sigma^* , \\[2mm]
{\bf 1}(z,0,\pi) &=& {\bf 1}\{ 0 \in z \} \,
{\bf 1}\{ h\in [z] \} \, {\bf 1}\{ s  =\min \{ s'>\beta (z) :\, (h',s') \in \pi \}\} \,.
\end{eqnarray*}
The two notations distinguish whether $z$  is a zero-cell or not. Both terms are equal to 1 if $z$ is divided by $h$ at time $s$. 

At first, because $m(\sigma)$ and $\underline{y}(\sigma)$ are uniquely determined by $\sigma$,  the transformation formula for image measures implies that
\begin{align*}
A &:=\int \sum_{(p,s)\in m} g( m \wedge z(p,s)  ,z(p,s) , p, s)\,\PP_{M} (\dint  m)\\
& =\int \sum_{(p,s)\in m(\sigma )} g(m( \sigma)\wedge z(p,s), z(p,s) , p, s)\,\PP_{\Sigma } (\dint \sigma )\,.
\end{align*}
A maximal polytope of dimension $d-1$ appears once a cell gets divided. Using the  rules from the global construction of the STIT tessellation, the definition of the process  $\Sigma$ as a sum of $\Sigma^*$ and $\delta_{(0,0,\Pi)}$ and the abbreviatory notation given above this leads to
\begin{eqnarray*}
A &=& \int \int \sum_{z \in \underline{y}( \sigma^* + \delta_{(0,0,\pi )})} \left(\sum_{(x,r,\psi )\in \sigma^*} \ \sum_{(h,s)\in \psi} {\bf 1}(z,\sigma^*,\psi) +\right. \\
&& + \left. \sum_{(h,s) \in \pi} {\bf 1}(z,0,\pi)   \right) g\big( m( \sigma^*+ \delta_{(0,0,\pi )})\wedge z ,z , z\cap h , s\big) \PP_\Pi(\dint \pi )\PP_{\Sigma^*} (\dint \sigma^* )
\end{eqnarray*}
Applying the  Mecke formula \eqref{eq:SlivnyakMecke} to the Poisson point processes $\Sigma^*$ yields
\begin{eqnarray*}
A &=& \int \int \int \int \int \sum_{z \in \underline{y}( \sigma^* + \delta_{(x,r,\psi )}+ \delta_{(0,0,\pi )})} \left( \sum_{(h,s)\in \psi} {\bf 1}(z,\sigma^*+ \delta_{(x,r,\psi )},\psi) +\right. \\
&& + \left. \sum_{(h,s) \in \pi} {\bf 1}(z,0,\pi)   \right) g\big( m( \sigma^*+ \delta_{(x,r,\psi )}+ \delta_{(0,0,\pi )})\wedge z ,z , z\cap h , s\big)\\
&& \PP_\Pi(\dint \psi )\ell_d(\dint x) \ell_-(\dint r)\PP_\Pi(\dint \pi )\PP_{\Sigma^*} (\dint \sigma^* ) \,.
\end{eqnarray*}
Next, we apply the Mecke formula \eqref{eq:SlivnyakMecke} again, this time twice to the Poisson point process $\Pi$, which has intensity measure $\L\otimes\ell_+$. This leads to the equation
\begin{align}\label{eq:Zwischenschritt1_Neu}
A = & \int\int\int\int\int [A_1+A_2]\,\PP_\Pi(\dint \psi ) \ell_d(\dint x) \ell_-(\dint r) \PP_\Pi(\dint \pi ) \PP_{\Sigma^*} (\dint \sigma^* )
\end{align}
with the terms $A_1$ and $A_2$ given by
\begin{align*}
A_1:=\int\int & \sum_{z \in  \underline{y}(\sigma^* + \delta_{(x,r,\psi+\delta_{(h,s)})} +\delta_{(0,0,\pi)} )} {\bf 1}(z,\sigma^* + \delta_{(x,r,\psi+\delta_{(h,s)})} ,\psi+\delta_{(h,s)})   \\ &   
\cdot g\big(m( \sigma^* + \delta_{(x,r,\psi +\delta_{(h,s)})} +\delta_{(0,0,\pi)} )\wedge z ,z , z\cap h , s\big)\, \Lambda (\dint h) \,  \dint s
\end{align*}
and
\begin{align*}
A_2:= \int \int &
 \sum_{z \in  \underline{y}(\sigma^* + \delta_{(x,r,\psi)} +\delta_{(0,0,\pi +\delta_{(h,s)})} )   } {\bf 1}(z,0,\pi +\delta_{(h,s)})\\ &   
\cdot g\big(m( \sigma^* + \delta_{(x,r,\psi )} +\delta_{(0,0,\pi +\delta_{(h,s)})} )\wedge z ,z , z\cap h , s\big)\, 
\Lambda (\dint h)  \dint s\,.
\end{align*}

Now,  notice that the $(x,r)$-value of $(x,r,\psi +\delta_{(h,s)} )$ is the same as that of $(x,r,\psi)$. Furthermore,
 $z \in  \underline{y}(\sigma^* + \delta_{(x,r,\psi+\delta_{(h,s)})} +\delta_{(0,0,\pi)} )  $ and the value of the indicator ${\bf 1}(z,\sigma^* + \delta_{(x,r,\psi+\delta_{(h,s)})} ,\psi+\delta_{(h,s)})$  in $A_1$ is 1, if and only if  $z \in  {y}(\sigma^* + \delta_{(x,r,\psi)} +\delta_{(0,0,\pi)} )_s $ and $z$ is divided by $h$ at time $s$. 
%\begin{align*}
%A_1=\int\int & \sum_{z \in  {y}(\sigma^* + \delta_{(x,r,\psi)} +\delta_{(0,0,\pi)} )_s   } {\bf 1}(z,\sigma^* + \delta_{(x,r,\psi)} ,\psi) \\ 
% &   
%\cdot g\big(m( \sigma^* + \delta_{(x,r,\psi +\delta_{(h,s)})} +\delta_{(0,0,\pi)} %)\wedge z ,z , z\cap h , s\big)\, \Lambda (\dint h)  \, \dint s
%\end{align*}
%and similarly
%\begin{align*}
%A_2= \int \int &
% \sum_{z \in  {y}(\sigma^* + \delta_{(x,r,\psi)} +\delta_{(0,0,\pi )} )_s   } {\bf 1}(z,0,\pi ) \\ 
%&   
%\cdot g\big(m( \sigma^* + \delta_{(x,r,\psi )} +\delta_{(0,0,\pi +\delta_{(h,s)})} )\wedge z ,z , z\cap h , s\big)\, 
%\Lambda (\dint h)  \dint s\,.
%\end{align*}

If  $s>0$, $z\in  y (\sigma )_s$ and $h\in [z]$, then $ \underline{y}( \sigma ,\oslash_{s,z,h}) $
denotes the realization of the STIT tessellation process which until time $s$ coincides with
$  \underline{y}( \sigma) $, at time $s$ the cell $z$ is divided by $h$, and after time $s$ the global construction is continued based on $\sigma $. Note that the division of $z$ by $h$ has an impact on the construction after time $s$. With this notation, it follows that
\begin{align*}
A_1 = & \int\int \sum_{z \in   y(\sigma^* + \delta_{(x,r,\psi )} +\delta_{(0,0,\pi)})_s }{\bf 1}(z,\sigma^* + \delta_{(x,r,\psi)} ,\psi) \\
&\qquad  \cdot g\big(m((\underline{y}(\sigma^*+ \delta_{(x,r,\psi )}+\delta_{(0,0,\pi)},\oslash_{s,z,h}))) \wedge z ,z , z\cap h , s\big) \, \Lambda (\dint h)   \dint s   
\end{align*}
and
\begin{align*}
A_2= \int \int &
 \sum_{z \in  {y}(\sigma^* + \delta_{(x,r,\psi)} +\delta_{(0,0,\pi )} )_s   }{\bf 1}(z,0,\pi ) \\ 
&   
\cdot g\big(m( (\underline{y}(\sigma^*+ \delta_{(x,r,\psi )}+\delta_{(0,0,\pi)},\oslash_{s,z,h})) )\wedge z ,z , z\cap h , s\big)\, 
\Lambda (\dint h)  \dint s\,.
\end{align*}
Plugging this into \eqref{eq:Zwischenschritt1_Neu} and applying then backwards the Mecke formula \eqref{eq:SlivnyakMecke} to the Poisson point processes $\Sigma^*$ (not to $\Pi$), we conclude that 
\begin{align*}
A=&\int\int\bigg[\int\int \bigg( \sum_{z \in  y(\sigma^* + \delta_{(0,0,\pi)})_s   }\quad \sum_{(x,r,\psi )\in  \sigma^*} {\bf 1}(z,\sigma^*  ,\psi) + {\bf 1}(z,0,\pi )\bigg)\\
& \quad \cdot g\big(m((\underline{y}(\sigma^*+ \delta_{(0,0,\pi)},\oslash_{s,z,h}))) \wedge z ,z , z\cap h , s) \,\Lambda (\dint h)  \dint s\bigg]\PP_{\Pi}(\dint\pi)\PP_{\Sigma^*}(\dint\sigma^*)\,.
\end{align*}
Note that in this expression for a fixed cell $z$ and a hyperplane $h$ holds $$\sum_{(x,r,\psi )\in  \sigma^*} {\bf 1}(z,\sigma^*  ,\psi) + {\bf 1}(z,0,\pi )={\bf 1}\{ h\in[z]\}.$$

 Now we use once more the transformation theorem for image measures and the fact that $m$ and $\underline{y}$ are uniquely determined by $\pi$ and $\sigma^*$. Moreover,  notice that the cell $z$ is divided for the first time at  $s$ using the hyperplane $h$ into two daughter cells and that within these two daughter cells two independent STIT tessellation processes are realized. This yields  
\begin{align*}
&A=\int\int\int\int\int   \sum_{z\in  y_s} g\big((z\cap h)\cup (  m^{(1)}_{(+s)} \wedge (z\cap h^+)) \cup (  m^{(2)}_{(+s)}  \wedge (z\cap h^-) ) ,    z,z\cap h, s \big) \\
&\hspace{5cm} \cdot{\bf 1}\{h\in[z]\}\, \Lambda (\dint h) \PP_{M} (\dint  m^{(1)})\PP_{M} (\dint  m^{(2)})\PP_{Y_s}(\dint  {y_s})\dint s \,. 
\end{align*}
Together with the definition \eqref{eq:problambda} of the probability measure $\L_z$ this finally leads to the  identity
\begin{align*}
A &=\int \int   \sum_{z\in  y_s}  \int \Big[\int \int 
g\big( (z\cap h)\cup (m^{(1)}_{(+s)} \wedge (z\cap h^+)) \cup (  m^{(2)}_{(+s)} \wedge (z\cap h^-) ),    z,z\cap h, s\big)\\
&\hspace{5cm} \PP_M(\dint m^{(1)}) \PP_M(\dint  m^{(2)})\Big]  \Lambda_z (\dint h)\Lambda ([z])  \PP_{Y_s}(\dint y_s)\, \dint s
\end{align*}
and the proof of the theorem is complete.\hfill $\Box$

\subsection{Proof of Proposition \ref{lem:kdaxpolintersectSTITlang}}

The purpose of the present subsection is to prove Proposition \ref{lem:kdaxpolintersectSTITlang}. This is prepared by the following technical lemma.
Let $\cF_{d-1}(z)$ denote the set of all facets (that is, faces of dimension $d-1$) of a polytope $z\in\cP_d$.

\begin{lemma}\label{lem:cellinclusion}
For all non-negative measurable functions $\tilde g : \cP _{d-1}\to \RR$ and $0<s_1<s_2$, we have that
\begin{eqnarray*}
&& \int \sum_{z_1\in y_{s_1}} \int \int \sum_{z_2\in  y_{s_2-s_1} \wedge (z_1\cap h_1^+)} 
\tilde g (z_1\cap h_1 \cap z_2) {\bf 1} \{ (z_2\cap h_1) \in \cF_{d-1}(z_2) \}  \\
&& \\
&& \qquad\qquad\qquad\qquad\qquad\qquad\cdot {\bf 1}\{h_1\in [z_1]\}\,\PP_{Y_{s_2-s_1}}(\dint y_{s_2-s_1}) \Lambda (\dint h_1)
\PP_{Y_{s_1}}(\dint y_{s_1}) \\
&& \\
&=& \int \sum_{z\in y_{s_2}} \int 
\tilde g (  z \cap h_1 ) {\bf 1}\{h_1\in [z]\} 
 \Lambda (\dint h_1)
\PP_{Y_{s_2}}(\dint y_{s_2})\,. 
\end{eqnarray*}
\end{lemma}

\begin{proof}
Assume that $h_1 \cap \stackrel{\circ}{z_1}\not= \emptyset$ and $z_2\in  y_{s_2-s_1} \wedge (z_1\cap h_1^+)$. Then $z_2\cap h_1 \in \cF_{d-1}(z_2)$ if and only if $z_2 \subset z_1$ and there is a cell $z\in y_{s_2-s_1} \wedge z_1$ such that  $z_2=z\cap h_1^+$ and  $z_2\cap h_1 =z\cap h_1\not= \emptyset$. Hence, using Fubini's theorem,
\begin{eqnarray*}
&& \int \sum_{z_1\in y_{s_1}} \int \int \sum_{z_2\in  y_{s_2-s_1} \wedge (z_1\cap h_1^+)} 
\tilde g (z_1\cap h_1 \cap z_2) {\bf 1} \{ (z_2\cap h_1) \in \cF_{d-1}(z_2) \}  \\
&& \\
&& \qquad\qquad\qquad\qquad\qquad\qquad\cdot {\bf 1}\{h_1\in [z_1]\}\,\PP_{Y_{s_2-s_1}}(\dint y_{s_2-s_1}) \Lambda (\dint h_1)
\PP_{Y_{s_1}}(\dint y_{s_1}) \\
&& \\
&=& \int \int \int \sum_{z_1\in y_{s_1}}  \sum_{z\in  y_{s_2-s_1} \wedge z_1} 
\tilde g (z\cap h_1 )  {\bf 1}\{h_1\in [z]\}  
 \Lambda (\dint h_1) \PP_{Y_{s_2-s_1}}(\dint y_{s_2-s_1}) 
\PP_{Y_{s_1}}(\dint y_{s_1}) \\
&& \\
&=& \int \sum_{z\in y_{s_2}} \int 
\tilde g (  z \cap h_1 )  {\bf 1}\{h_1\in [z]\} 
 \Lambda (\dint h_1)
\PP_{Y_{s_2}}(\dint y_{s_2}), 
\end{eqnarray*}
where the last equality follows from (\ref{iterate}).
\end{proof}

Now we prove Proposition \ref{lem:kdaxpolintersectSTITlang}.
If $\left( \overline{\bf p},{\bf s}\right) =\left( \bigcap_{i=1}^{d-k} p_i ,{\bf s}\right)$ is a marked maximal $k$-polytope generated by a $(d-k)$-tuple $((p_1,s_1),\ldots (p_{d-k},s_{d-k}))=({\bf p}, {\bf s},k)\in m_t^{d-k}$,
then we can represent it in the following way which will be used in the formulas below. The $(d-1)$-polytope $p_1$ is located on a  hyperplane $h_1$ with birth time $s_1$, and at that time it divides a cell $z_1$, i.e., $p_1=z_1 \cap h_1$. For a STIT tessellation process, on both cells (indicated by + and --) adjacent to $p_1$ appear independent traces until time $s_2$  and these two traces will be treated  separately. Let us consider the case that the remaining maximal polytopes $((p_2,s_2),\ldots (p_{d-k},s_{d-k}))$ are located in the cell $z_1 \cap p_1^+$.
This cell is subdivided in the time interval $(s_1,s_2)$ by $\{ (p,s)\in m_t \wedge 
(z_1 \cap p_1^+): s_1<s<s_2 \}$. Then, at time $s_2$,  one of the cells, $z_2\subseteq z_1 \cap p_1^+$  is divided by $(p_2,s_2)$, and $\dim(p_1 \cap p_2)=d-2$. In particular, this means that one of the $(d-1)$-dimensional faces of $z_2$ is a subset of $p_1$, and this face is divided by $p_2$.
 The maximal $(d-1)$-polytope $p_2$ is located on a hyperplane $h_2$ with birth time $s_2$.
 
 This can now be continued inductively.  The combination of the possible choices in each step of the adjacent cells, indicated by + and --,  leads to a factor $2^{d-k-1}$. The $(d-k)$-tuple $({\bf p}, {\bf s},k)\in m_t^{d-k}$ will be processed step by step, and after separating $({p_1, s_1})$ the remaining $(d-k-1)$-tuple is denoted $({\bf p^{1}, s^{1}})$,  and so on. Further, assume $p_i\subset h_i \in \cH$, i.e., the hyperplane $h_i$ supports $p_i$.

For $0<s<t$ and $m$ a realization of $M$ denote
$m_{(+s,t)}:=\{ (p, s'+s):\ (p,s')\in M,\, 0<s'<t-s \}$, that is, the set of all birth time marked maximal $(d-1)$-polytopes, with a birth time shifted by $s$, and such that the shifted birth time is between $s$ and $t$. Furthermore, we denote
$\overline{\bf p}^j := \bigcap_{i=j+1}^{d-k} p_i$, for $j=1,\ldots d-k-2$. Now,
\begin{eqnarray*}
A&:=& \int \displaystyle {\sum_{({\bf p}, {\bf s},k)\in m_t^{d-k} }}
g\left( \overline{\bf p}, {\bf s},  m_t  \sqcap \ \overline{\bf p} \right) 
\PP_{{M_t}} (\dint  m_t) 
\\ &&\\
&=& \int \displaystyle {\sum_{({p_1, s_1})\in m_t}} \ \ {\sum_{({\bf p^{1}, s^{1}})}}
g\left( \overline{\bf p}, {\bf s} ,  m_t  \sqcap \ \overline{\bf p} \right) 
\PP_{{M_t}} (\dint  m_t).
\end{eqnarray*}
Next, we apply Theorem \ref{theor:slivnyakgeneratorSTIT}, exchange the order of integration and  partition the sum into two parts. This yields
\begin{eqnarray*} 
A&=& \int\int\int\int \int \sum_{z_1\in  y_{s_1}}   
\\  &&   \\
&& \Bigg[ 
{\sum_{({\bf p^{1}, s^{1}})\in ({m_{(+s_1,t)}^{+,(d-k-1)}}\wedge (z_1\cap h_1^+))}}   {\bf 1}\{ \dim(  z_1\cap h_1 \cap \overline{\bf p}^1 )=k \} 
\\  &&   \\
&& \cdot g\left(   z_1\cap h_1 \cap \overline{\bf p}^1, {\bf s}, 
         [(m_{(+s_1,t)}^-  \cup m_{(+s_1,t)}^+ ) \setminus \{  p_2, \ldots , p_{d-k-1} \} ] \cap z_1\cap h_1  \cap  \overline{\bf p}^1 \right)
 \\ &&   \\
&&  + 
{\sum_{({\bf p^{1}, s^{1}})\in ({m_{(+s_1,t)}^{-,(d-k-1)}}\wedge (z_1\cap h_1^-))}}
{\bf 1}\{ \dim(  z_1\cap h_1 \cap \overline{\bf p}^1 )=k \}
\\  &&   \\
&& 
 \cdot g\left(  z_1\cap h_1 \cap \overline{\bf p}^1,  {\bf s} , 
         [(m_{(+s_1,t)}^-  \cup m_{(+s_1,t)}^+ ) \setminus \{  p_2, \ldots , p_{d-k-1} \} ] \cap z_1\cap h_1  \cap \ \overline{\bf p}^1 \right) \Bigg]
         \\ &&  \\
&& 
         \PP_{{M}} (\dint  m^-)  \PP_{{M}} (\dint  m^+) 
 {\bf 1}\{h_1\in [z_1]\}   \Lambda (\dint h_1)  {\bf 1}\{ 0<s_1< t \}
P_{ Y_{s_1}} (\dint  y_{s_1})  \dint s_1\,.
\end{eqnarray*}
In the first item in squared brackets, i.e., the case 
$({\bf p^{1}, s^{1}})\in ({m_{(+s_1,t)}^{+,(d-k-1)}}\wedge (z_1\cap h_1^+))$, 
decompose $({\bf p^{1}, s^{1}})$ into $(p_2,s_2)$ and the remaining $(d-k-2)$-tuple $({\bf p^{2}, s^{2}})$. Applying Theorem \ref{theor:slivnyakgeneratorSTIT} once again, but this time to
$$
\int \displaystyle{\sum_{(p_2, s_2)\in (m_{(+s_1,t)}^+ \wedge (z_1\cap h_1^+))}} \{\ldots\}\, \PP_{{M}} (\dint  m^+)
$$
and noting that $z_2\subset z_1$, yields
\begin{align*}
& \int \int {\sum_{({\bf p^{1}, s^{1}})\in ({m_{(+s_1,t)}^{+,(d-k-1)}}\wedge (z_1\cap h_1^+))}}   {\bf 1}\{ \dim( \overline{\bf p}^1 \cap z_1\cap h_1 )=k \} \\ 
& \cdot g\left(   z_1\cap h_1 \cap \overline{\bf p}^1, {\bf s}, 
         [(m_{(+s_1,t)}^-  \cup m_{(+s_1,t)}^+ ) \setminus \{ z_1\cap h_1, p_2, \ldots , p_{d-k-1} \} ] \cap z_1\cap h_1  \cap  \overline{\bf p}^1 \right)\\
&  \qquad\qquad\qquad  \PP_{{M}} (\dint  m^-)  \PP_{{M}} (\dint  m^+) \\
&= \int \ldots \int \sum_{z_2\in  y_{s_2-s_1} \wedge (z_1\cap h_1^+)}   \\ 
& \Bigg[ 
{\sum_{({\bf p^{2}, s^{2}})\in  (m_{(+s_2,t)}^{++,(d-k-2)}\wedge (z_2\cap h_2^+))}}   {\bf 1}\{ \dim( \overline{\bf p}^2 \cap z_2\cap h_1 \cap h_2 )=k \} \\  
& \cdot g\Big(   z_2\cap h_1 \cap h_2 \cap \overline{\bf p}^2, {\bf s}, 
         [(m_{(+s_1,t)}^-  \cup m_{(+s_2,t)}^{++} \cup m_{(+s_2,t)}^{+-} ) \setminus \{  p_3, \ldots , p_{d-k-1} \} ]\cap\ldots\\
 &\hspace{5cm}       \ldots \cap z_2\cap h_1 \cap h_2 \cap  \overline{\bf p}^2 \Big)\\
&  + 
{\sum_{({\bf p^{2}, s^{2}})\in  (m_{(+s_2,t)}^{+-,(d-k-2)}\wedge (z_2\cap h_2^-))}}
{\bf 1}\{ \dim( \overline{\bf p}^2 \cap z_2\cap h_1 \cap h_2 )=k \}\\
& \cdot g\Big(  z_2\cap h_1 \cap h_2 \cap \overline{\bf p}^2,  {\bf s} , 
         [(m_{(+s_1,t)}^-  \cup m_{(+s_2,t)}^{++} \cup m_{(+s_2,t)}^{+-} ) \setminus \{  p_3, \ldots , p_{d-k-1} \} ]\cap\ldots\\
&\hspace{5cm}         \ldots \cap z_2\cap h_1  \cap h_2 \cap  \overline{\bf p}^2 \Big) \Bigg]\\ 
&\qquad\qquad\qquad\PP_{{M}} (\dint  m^-) \PP_{{M}} (\dint  m^{++})  \PP_{{M}} (\dint  m^{+-})  \PP_{Y_{s_2-s_1}}(\dint y_{s_2-s_1})        \\ 
&\qquad\qquad\qquad\cdot {\bf 1}\{h_2\in [z_2]\}   \Lambda (\dint h_2)  {\bf 1}\{ 0< s_1< s_2<t \} \dint s_2 \,.
\end{align*}

Now, we apply this argument repeatedly to all summands and decompose $({\bf p, s})$ step by step. Note that $z_{d-k}\subset \ldots \subset z_1  $, and also that the intersections like $[(m_{(+s_1,t)}^-  \cup m_{(+s_2,t)}^{++} \cup m_{(+s_2,t)}^{+-} ) \setminus \{  p_3, \ldots , p_{d-k-1} \} ] \cap z_2\cap h_1  \cap h_2 \cap  \overline{\bf p}^2$ do not depend on the combinations of signs (which determine a part of the space) in the upper index. 
Hence, Lemma \ref{lem:cellinclusion} yields
\begin{eqnarray*}
A &=& \sum_{(a_1,\ldots , a_{d-k-1})\in \{ +,-\}^{d-k-1 }} \int \ldots \int
\sum_{z_{d-k}\in  y_{s_{d-k}}}
\\ &&\\
&& 
\cdot g\left(   z_{d-k}\cap \overline{\bf h}, {\bf s},
\left( \bigcup_{i=1}^{d-k-1} m_{(+s_i,t)}^{(i)}  \cup m_{(+s_{d-k},t)}^{+} \cup m_{(+s_{d-k},t)}^{-} \right) \cap z_{d-k} \cap \overline{\bf h}  
\right)
\\ &&\\
&& \cdot {\bf 1}\left\{  z_{d-k}\cap \overline{\bf h}\not= \emptyset \right\} \, \Lambda (\dint h_{d-k})
\ldots \Lambda (\dint h_2) \Lambda (\dint h_1)
\\ &&\\
&&
\PP_{M} (\dint  {m}^{+}) \PP_{M} (\dint  {m}^{-}) 
\PP_{M} (\dint  {m}^{(d-k-1)})\ldots \PP_{ M} (\dint  { m}^{(1)}) \PP_{Y_{s_{d-k}}} (\dint {y_{s_{d-k}}})    
\\ &&\\
&& 
 \cdot{\bf 1} \{ 0< s_1<\ldots < s_{d-k}<t\}\,  \dint s_1 \ldots \dint s_{d-k} \, .
 \end{eqnarray*}
Now substitute  $ m_{(+s_i,t)}^{(i)}$ by the corresponding STIT tessellations $y_{t-s_i}^{(i)}=y( m_{(+s_i,t)}^{(i)})$. Furthermore, due to the spatial consistency of STIT tessellations the values of the summands do not depend on $(a_1,\ldots , a_{d-k-1})\in \{ +,-\}^{d-k-1 }$. Noting finally, that the first sum is running over $2^{d-k-1}$ terms leads to the identity
\begin{align*}
& \int \displaystyle {\sum_{({\bf p}, {\bf s},k)\in m_t^{d-k}} }
g\left( \overline{\bf p}, {\bf s},  m_t  \sqcap \ \overline{\bf p} \right) 
\PP_{{M_t}} (\dint  m_t) \\
&=   2^{d-k-1} 
\int \ldots \int \sum_{z\in  y_{s_{d-k}}} \textstyle{
g\Big(  z\cap \overline{\bf h},{\bf s} , 
 z\cap \overline{\bf h} \cap \left[ \bigcup_{i=1}^{d-k-1} \partial y_{t -s_{i}}^{(i)} \cup \partial y_{t -s_{d-k}}^{+} 
 \cup \partial y_{t -s_{d-k}}^{-} \right] \Big) \  }\\ 
&   \qquad\qquad\qquad\PP_{\underline Y}^{\otimes (d-k+1)} (\dint ({\underline y^{(1)}}, \ldots 
{\underline y^{(d-k-1)}},{\underline y^{+}},{\underline y^{-}})){\bf 1}\left\{  z\cap \overline{\bf h}\not= \emptyset \right\} \, \Lambda^{\otimes (d-k)} (\dint {\bf h})\\ 
&  \qquad\qquad\qquad  \PP_{Y_{s_{d-k}}} (\dint {y_{s_{d-k}}})     \cdot 
{\bf 1} \{ 0< s_1<\ldots < s_{d-k}<t\}\,  \dint s_1 \ldots \dint s_{d-k}   \,,
  \end{align*}
which completes the proof.\hfill $\Box$  
  
\subsection{Proof of Proposition \ref{lem:rho1}}

Using the scaling property \eqref{eq:STITglobalScaling}, changing the order of integration and substituting  $y_{s_{d-k}}$ by 
$\frac{1}{{s_{d-k}}} y_1$, we obtain from \eqref{eq:rhoformel} that
\begin{align*}
 \r_{k,1}^{(j)}  
&=  2^{d-k-1} \frac{1}{\ell_d(B)} \int \int \int \sum_{z\in  y_{s_{d-k}}} \textstyle{
{\bf 1}\{ c(z\cap \overline{\bf h})\in B \} V_j(z\cap \overline{\bf h})  }  \\
& \qquad \PP_{ Y_{s_{d-k}}} (\dint { y_{s_{d-k}}})  \Lambda^{\otimes (d-k)} (\dint {\bf h})  
{\bf 1} \{ 0<  s_{d-k}< 1\}\,  \frac{s_{d-k}^{d-k-1}}{(d-k-1)!} \dint s_{d-k}\\
&=  2^{d-k-1} \frac{1}{\ell_d(B)} \int \int \int \sum_{z\in  \frac{1}{{s_{d-k}}} y_1} \textstyle{
{\bf 1}\{ c(z\cap \overline{\bf h})\in B \} V_j(z\cap \overline{\bf h})  }  \\
& \qquad  \PP_{ Y_{1}} (\dint  y_{1}) \Lambda^{\otimes (d-k)} (\dint {\bf h})  
{\bf 1} \{ 0<  s_{d-k}< 1\}\,  \frac{s_{d-k}^{d-k-1}}{(d-k-1)!} \dint s_{d-k}\\
&=  2^{d-k-1} \frac{1}{\ell_d(B)} \int \int \int \sum_{z\in   y_1} \textstyle{
{\bf 1}\{ c(\frac{1}{{s_{d-k}}}z\cap \overline{\bf h})\in B \} V_j(\frac{1}{{s_{d-k}}}z\cap \overline{\bf h})  }  \\
& \qquad  \PP_{ Y_{1}} (\dint  y_{1}) \, \Lambda^{\otimes (d-k)} (\dint {\bf h})
{\bf 1} \{ 0<  s_{d-k}< 1\}\,  \frac{s_{d-k}^{d-k-1}}{(d-k-1)!} \dint s_{d-k}\,.
\end{align*}
We consider the two inner integrals separately.
Let $\gamma_1$ denote the mean number of cell centroids per unit volume and let ${\mathbb Q}_1$ denote the distribution of the typical cell of $Y_1$. Then an application of Campbell's theorem, multiplication with $s_{d-k}$, and the homogeneity of of the $j$th intrinsic volume $V_j$ yield 
\begin{align*}
I &:=\frac{1}{\ell_d(B)} \int \int \sum_{z\in   y_1} \textstyle{
{\bf 1}\{ c(\frac{1}{{s_{d-k}}}z\cap \overline{\bf h})\in B \} V_j(\frac{1}{{s_{d-k}}}z\cap \overline{\bf h})  }  \,
 \PP_{ Y_{1}} (\dint  y_{1})  \Lambda^{\otimes (d-k)} (\dint {\bf h})\\
&= \frac{1}{\ell_d(B)} \gamma_1 \int  \int \int 
{\bf 1}\{ c((z+x)\cap s_{d-k}\overline{\bf h})\in s_{d-k} B \} \, s_{d-k}^{-j} \\
&\qquad\qquad\qquad\qquad \cdot V_j((z+x)\cap s_{d-k}\overline{\bf h})   \, \ell_d (\dint x)  {\mathbb Q}_1 (\dint z)\Lambda^{\otimes (d-k)} (\dint {\bf h})\,.
\end{align*}
In a next step, we use that $s_{d-k}\overline{\bf h}=\overline{\bf h}+ (s_{d-k}-1)x^\perp = \overline{\bf h}_0 + s_{d-k}x^\perp$, which is a translation of $\overline{\bf h}$, where $x^\perp = \overline{\bf h}\cap \overline{\bf h}_0^\perp$, $\overline{\bf h}_0$ the $k$-dimensional linear subspace parallel to $\overline{\bf h}$ and $\overline{\bf h}_0^\perp$ its orthogonal complement.   The image of the measure 
${\bf 1}\{ \dim (\overline{\bf h})=k \}\cdot \Lambda^{\otimes (d-k)} (\dint {\bf h})$ the product measure, endowed with the indicator density) under the mapping ${\bf h} \mapsto \overline{\bf h}$ is invariant under translations. Then, according to \cite[Theorem 4.4.1]{schn/weil}, we obtain that 
\begin{align*}
I=&\frac{1}{\ell_d(B)} \gamma_1 \int  \int \int \textstyle{ 
{\bf 1}\{ c((z+x)\cap \overline{\bf h})\in s_{d-k} B \} \, s_{d-k}^{-j} \, V_j((z+x)\cap \overline{\bf h})  } \\
&\qquad\qquad\qquad\qquad\qquad \ell_d (\dint x)  {\mathbb Q}_1 (\dint z)s_{d-k}^{-(d-k)} \Lambda^{\otimes (d-k)} (\dint {\bf h}) \\ 
&= s_{d-k}^{d}\frac{1}{\ell_d(s_{d-k} B)} \gamma_1 \int  \int \int \textstyle{ 
{\bf 1}\{ c((z+x)\cap \overline{\bf h})\in s_{d-k} B \} \, s_{d-k}^{-j} \, V_j((z+x)\cap \overline{\bf h})  } \\
&\qquad\qquad\qquad\qquad\qquad \ell_d (\dint x)  {\mathbb Q}_1 (\dint z)s_{d-k}^{-(d-k)} \Lambda^{\otimes (d-k)} (\dint {\bf h}) \\ 
&= s_{d-k}^{k-j}\frac{1}{\ell_d(s_{d-k} B)} \gamma_1 \int  \int \int \textstyle{ 
{\bf 1}\{ c((z+x)\cap \overline{\bf h})\in s_{d-k} B \}   \, V_j((z+x)\cap \overline{\bf h})  } \\
&\qquad\qquad\qquad\qquad\qquad \ell_d (\dint x)  {\mathbb Q}_1 (\dint z) \Lambda^{\otimes (d-k)} (\dint {\bf h}) \\ 
&=s_{d-k}^{k-j}\frac{1}{\ell_d(B)} \int \int \sum_{z\in   y_1} \textstyle{
{\bf 1}\{ c(z\cap \overline{\bf h})\in B \} V_j(z\cap \overline{\bf h})  }  
 \PP_{ Y_{1}} (\dint  y_{1})  \Lambda^{\otimes (d-k)} (\dint {\bf h})\,,
\end{align*}
where the last equation follows from Campbell's theorem and by replacing $s_{d-k} B$ by $B$.

Plugging this expression for $I$ into the equation for $\r_{k,1}^{(j)}$ above, yields
\begin{align*}
 \r_{k,1}^{(j)}&=  2^{d-k-1} \frac{1}{\ell_d(B)} \int \int \int \sum_{z\in   y_1} \textstyle{
s_{d-k}^{k-j}{\bf 1}\{ c(z\cap \overline{\bf h})\in B \}  V_j(z\cap \overline{\bf h})  }  \\
& \qquad  \PP_{ Y_{1}} (\dint  y_{1})  \Lambda^{\otimes (d-k)} (\dint {\bf h})
{\bf 1} \{ 0<  s_{d-k}< 1\}\,  \frac{s_{d-k}^{d-k-1}}{(d-k-1)!} \dint s_{d-k}\\
&=
2^{d-k-1}  \frac{1}{(d-k-1)!\, (d-j)} \frac{1}{\ell_d(B)}\\
&\qquad\cdot  \int  \int \sum_{z\in  y_1} \textstyle{
{\bf 1}\{ c(z\cap \overline{\bf h})\in B \} V_j(z\cap \overline{\bf h})  }  
  \PP_{ Y_{1}} (\dint { y_1})  \Lambda^{\otimes (d-k)} (\dint {\bf h})\,,
\end{align*}
where the last equation results by integration with respect to $s_{d-k}$.  \hfill $\Box$

 \subsection{Proof of Theorem \ref{thm:GeneralBirthTime}}
 
For any  non-negative measurable function $g: (0,t)^{d-k}\to \RR$, Corollary \ref{lem:kdaxpolintersectSTIT} and an application of (\ref{eq:maxpolbirth}) yield
 \begin{align*}
 &\int g( {\bf s}) \, \QQ^{(j)}_{{\boldsymbol{ \beta}},t} (\dint {\bf s}) \\ 
&= \int g( {\bf s}) \, \QQ^{(j)}_{(\overline{\bf P}, {\boldsymbol{ \beta}},\tau ),t} (\dint (q, {\bf s}, T)) \\ 
&= 2^{d-k-1} \left[ \r_{k,t}^{(j)}\right]^{-1} \int \ldots \int \sum_{z\in  y_{s_{d-k}}} \textstyle{ {\bf 1} \{ c(z\cap \overline{\bf h})\in [0,1]^d \} \cdot
 V_j (z\cap \overline{\bf h} )\cdot  g(  {\bf s})\  } \Lambda^{\otimes (d-k)} (\dint {\bf h}) \\ 
& \qquad\cdot {\bf 1} \{ 0< s_1<\ldots < s_{d-k} <t\}\,  \dint s_1 \ldots \dint s_{d-k-1} \,   \PP_{ Y_{s_{d-k}}} (\dint { y_{s_{d-k}}}) \dint s_{d-k}\,.
\end{align*}   
Using (\ref{eq:STITglobalScaling}) and substituting  $y_{s_{d-k}}$ by 
$\frac{1}{{s_{d-k}}} y_1$ we obtain (similarly to the calculations in the proof of Proposition \ref{lem:rho1})
 \begin{align*}
 & \int g( {\bf s}) \,
 \QQ^{(j)}_{{\boldsymbol{ \beta}},t} (\dint {\bf s}) \\ 
&=  2^{d-k-1} \left[ \r_{k,t}^{(j)}\right]^{-1} 
\int \ldots \int \sum_{z\in  \frac{1}{{s_{d-k}}} y_1} \textstyle{ {\bf 1} \{ c(z\cap \overline{\bf h})\in [0,1]^d \}\cdot
V_j (z\cap \overline{\bf h} )\cdot g(  {\bf s})\  }\\ 
&\qquad \Lambda^{\otimes (d-k)} (\dint {\bf h}) 
 \PP_{ Y_1} (\dint { y_1}) {\bf 1} \{ 0< s_1<\ldots < s_{d-k} <t\}\,  \dint s_1 \ldots \dint s_{d-k-1}\dint s_{d-k}  \\ 
&=  2^{d-k-1}\left[ \r_{k,t}^{(j)}\right]^{-1} 
\int  \int \sum_{z\in   y_1} \textstyle{  {\bf 1} \{c( z\cap \overline{\bf h})\in [0,1]^d\}  \cdot
 V_j ( z\cap \overline{\bf h} )} \, \Lambda^{\otimes (d-k)} (\dint {\bf h}) 
 \PP_{ Y_1} (\dint { y_1}) \\ 
&\qquad\cdot \int  \ldots \int g(  {\bf s})\cdot s_{d-k}^{k-j} \cdot {\bf 1} \{ 0< s_1<\ldots < s_{d-k} <t\}\,  \dint s_1 \ldots \dint s_{d-k-1}\dint s_{d-k}\,.
\end{align*}  
We can now use Proposition \ref{lem:rho1} together with the scaling relation \eqref{eq:ScalingRho} to evaluate $\r_{k,t}^{(j)}$. This  yields the desired result immediately. \hfill $\Box$

\subsection{Proof of Corollary \ref{lem:maxpolbirth}}

We have
\begin{align*}\label{eq:maxpolbirthmargin}
&\int g( q, s_{d-k}) \,
 \QQ^{(j)}_{(\overline{\bf P}, { \beta_{d-k}} ),t} (\dint (q,  s_{d-k})) \\
&= 
 2^{d-k-1} \left[ \r_{k,t}^{(j)}\right]^{-1}\\ 
&\qquad\cdot \int \int \int \sum_{z\in  y_{s_{d-k}}} \textstyle{
V_j(z\cap \overline{\bf h})\cdot{\bf 1}\{c(z\cap \overline{\bf h})\in {[0,1]^d}\} \cdot  g((z\cap \overline{\bf h})-c(z\cap \overline{\bf h}),s_{d-k})\  } 
\\ 
&\qquad\qquad  \Lambda^{\otimes (d-k)} (\dint {\bf h}) \PP_{Y_{s_{d-k}}} (\dint { y_{s_{d-k}}}) 
{{s_{d-k}^{d-k-1}}\over {(d-k-1)!}  } {\bf 1} \{ 0<s_{d-k}<t\}\,   \dint s_{d-k} \,   
\\ 
&= 
 2^{d-k-1} \left[ \r_{k,t}^{(j)}\right]^{-1}\\
& \qquad\cdot \int \int \int \sum_{z\in  y_{s_{d-k}}} \textstyle{
V_j(z\cap \overline{\bf h})\cdot {\bf 1}\{c(z\cap \overline{\bf h})\in {[0,1]^d}\})\cdot  g(  (z\cap \overline{\bf h})-c(z\cap \overline{\bf h}),s_{d-k})\  } 
\\ 
&\qquad\qquad  \Lambda^{\otimes (d-k)} (\dint {\bf h}) \PP_{ Y_{s_{d-k}}} (\dint { y_{s_{d-k}}})  
{{s_{d-k}^{-(k-j)}}\over {(d-k-1)! (d-j)}  }     \,   
t^{d-j}\, \QQ^{(j)}_{\beta_{d-k},t}( \dint s_{d-k})
 \end{align*}
 and the result follows. \hfill $\Box$

\subsection{Proof of Theorem \ref{markovbirthCL}}

It is sufficient to consider functions  $g$ of the form
$$
g( q, {\bf s})=g_1( q)\cdot g_2(s_1,\ldots ,s_{d-k-1}) \cdot g_3(s_{d-k})\,,
$$
where $g_1: \cP_k \to \RR$, $g_2 : (0,t)^{d-k-1}\to \RR$, $g_3: (0,t)\to \RR$ are non-negative measurable functions. The proposition for general $g$ follows then by a standard measure-theoretic procedure. As in the proof of Theorem \ref{thm:GeneralBirthTime}, we have
\begin{align*} 
& \int  g( q, {\bf s}) \QQ^{(j)}_{(\overline{\bf P}, {\beta_1,\ldots , \beta_{d-k}} ),t} ( \dint (q, {\bf s})) \\
&=
 \left[ \r_{k,t}^{(j)}\right]^{-1}
 \int 
 \sum_{({\bf p}, {\bf s},k)\in m_t^{d-k}}
 {\bf 1}\{ c(\overline{\bf p})\in [0,1]^d\}\,  V_j(\overline{\bf p})\, g\left( \overline{\bf p}-c(\overline{\bf p}), {\bf s} \right)  \PP_{ M_t} (\dint m_t) 
  \\ 
&=  2^{d-k-1}\left[ \r_{k,t}^{(j)}\right]^{-1} \\ 
&\qquad\cdot\int \ldots \int \sum_{z\in   y_{s_{d-k}} } g_1\left(  (z\cap \overline{\bf h})-c(z\cap \overline{\bf h}))  \right)\, 
  {\bf 1}\{ c( z\cap \overline{\bf h})\in [0,1]^d\} \,
 V_j ( z\cap \overline{\bf h} )\,  
  \\ 
& \qquad\qquad \Lambda^{\otimes (d-k)} (\dint {\bf h}) 
 \PP_{ Y_{s_{d-k}}} (\dint y_{s_{d-k}}) g_2(s_1,\ldots ,s_{d-k-1}) \, {\bf 1} \{ 0< s_1<\ldots < s_{d-k} \}\,  \dint s_1 \ldots \dint s_{d-k-1}\\ 
& \qquad\qquad\qquad\cdot g_3(s_{d-k})\, {\bf 1} \{ 0< s_{d-k} <t\}\dint s_{d-k}\,.
\end{align*} 
Now, we apply Corollary \ref{cor:marginalconditional} and  Corollary \ref{lem:maxpolbirth} and obtain that this is equal to
\begin{align*} 
&
\int \int \int \frac{(d-k-1)! (d-j) s_{d-k}^{k-j}}{t^{d-j}}   g_1 (q) \,\QQ^{(j)}_{{\overline{\bf P} ,t}| \beta_{d-k}=s_{d-k} } (\dint q) 
   \\ 
& \qquad\cdot g_2(s_1,\ldots ,s_{d-k-1}) \,   
{{1}\over{(d-k-1)!}}\,  s_{d-k}^{d-k-1} 
\QQ^{(j)}_{{( {\beta_1,\ldots , \beta_{d-k-1}} ),t} |\beta_{d-k}=s_{d-k} } \\ 
& \qquad\qquad \dint (s_1,\ldots , s_{d-k-1})\,g_3(s_{d-k})\, {t^{d-j}\over (d-j)} s_{d-k}^{-(d-j-1)}\QQ^{(j)}_{\beta_{d-k}} (\dint s_{d-k} )
 \\ 
&= \int \int \int  g(q, {\bf s} ) \QQ^{(j)}_{{\overline{\bf P} ,t}| \beta_{d-k}=s_{d-k} } (\dint q)
  \QQ^{(j)}_{{( {\beta_1,\ldots , \beta_{d-k-1}} ),t} |\beta_{d-k}=s_{d-k} } \\
&\qquad\qquad\qquad \dint (s_1,\ldots , s_{d-k-1})  \QQ^{(j)}_{\beta_{d-k}} (\dint s_{d-k} )\,,
\end{align*} 
which completes the proof.\hfill $\Box$
  
\subsection{Proof of Theorem \ref{thm:GeneralInternalVertices}}\label{sec:proofs2}

For fixed ${\bf h}=(h_1,\ \ldots , h_{d-1})\in \cH ^{d-1}$ with the hyperplanes in general position define the line $\overline{\bf h}=\bigcap_{i=1}^{d-1} h_{i}$.
Because the intersection of a STIT with a line is a Poisson point process (see \cite{nagel/weiss03}),   $\overline{\bf h} \cap \left[ \bigcup_{i=1}^{d-2} \partial y_{t -s_{i}}^{(i)} \cup \partial y_{t -s_{d-1}}^{+} 
 \cup \partial y_{t -s_{d-1}}^{-} \right] $ is a realization of a superposition of $d$ independent Poisson point processes on the line $\overline{\bf h}$, with a law invariant under translations on this line. Due to the stationarity of STIT tessellations, the intensity of this point process depends only on the direction of this line, which we denote by $u\in \cS^{d-1}$, and,  up to a factor $b(u)>0$, it is given by the sum
 $$
 a({\bf s})=\sum_{i=1}^{d-2}( t -s_{i}) + 2(t -s_{d-1})= d\cdot t - 2s_{d-1} -\sum_{i=1}^{d-2} s_{i}\,.
 $$
Thus for any cell $z$ the number of points of 
$$z\cap \overline{\bf h} \cap \left[ \bigcup_{i=1}^{d-2} \partial y_{t -s_{i}}^{(i)} \cup \partial y_{t -s_{d-1}}^{+} 
 \cup \partial y_{t -s_{d-1}}^{-} \right] $$
follows a Poisson distribution with parameter $ V_1 ( z\cap \overline{\bf h})\cdot b(u) \cdot a({\bf s})$. Now, we apply this fact together with \eqref{eq:maxpolbirth} and Proposition \ref{lem:kdaxpolintersectSTITlang} to conclude that for $j=0,1$
\begin{eqnarray}\label{eq:punktzahl}
&& \int {\bf 1} \{ \# T =n\} 
 \QQ^{(j)}_{(\overline{\bf P}, {\boldsymbol{ \beta}},\tau ),t} (\dint (\overline{\bf p}, {\bf s},T)) \nonumber \\ && \nonumber  \\
& = & \left[ \r_{1,t}^{(j)}\right]^{-1} \int \sum_{({\bf p}, {\bf s},1)\in m_t^{d-1}} 
V_j(\overline{\bf p})\cdot {\bf 1}\{ c(\overline{\bf p})\in [0,1]^d\}  {\bf 1} \{ \# (m_t  \sqcap \ \overline{\bf p})=n \} \PP_{ M_t} (\dint m_t) \nonumber \\ && \nonumber \\
&= & 
2^{d-2} \left[ \r_{1,t}^{(j)}\right]^{-1} \int \ldots \int \sum_{z\in  y_{s_{d-1}}} \textstyle{
\, 
V_j(z\cap \overline{\bf h})\, {\bf 1}\{ c( z\cap \overline{\bf h})\in [0,1]^d\} } \nonumber \\ && \nonumber \\
 &&  \qquad\cdot{\bf 1} \left\{ \# 
\left( z\cap \overline{\bf h} \cap \left[ \bigcup_{i=1}^{d-2} \partial y_{t -s_{i}}^{(i)} \cup \partial y_{t -s_{d-1}}^{+} 
 \cup \partial y_{t -s_{d-1}}^{-} \right] \right) =n \right\} 
\nonumber  \\ && \nonumber \\ 
&&  \qquad\qquad \PP_{\underline Y}^{\otimes d} (\dint ({\underline y^{(1)}}, \ldots 
{\underline y^{(d-2)}},{\underline y^{+}},{\underline y^{-}})) \, \PP_{ Y_{s_{d-1}}} (\dint y_{s_{d-1}})  \\
&&  \qquad\cdot {\bf 1} \{ 0< s_1<\ldots < s_{d-1}<t\}\,  \dint s_1 \ldots \dint s_{d-1} \,   
\Lambda^{\otimes (d-1)} (\dint {\bf h}) \nonumber  \\ 
&& \nonumber \\
&= & 
2^{d-2} \left[ \r_{1,t}^{(j)}\right]^{-1} \int \ldots \int \sum_{z\in  y_{s_{d-1}}} \textstyle{
 \, 
V_j(z\cap \overline{\bf h})\, {\bf 1}\{ c( z\cap \overline{\bf h})\in [0,1]^d\}} \nonumber \\ && \nonumber \\
 &&  \qquad\cdot \frac{[V_1 ( z\cap \overline{\bf h}) b(u)  a({\bf s})]^n}{n!}
  e^{-V_1 ( z\cap \overline{\bf h}) b(u)  a({\bf s})}
 \PP_{{ Y}_{s_{d-1}}} (\dint { y}_{s_{d-1}}) \\ &&\nonumber  \\
 &&
 \qquad\cdot{\bf 1} \{ 0< s_1<\ldots < s_{d-1}<t\}\,  \dint s_1 \ldots \dint s_{d-1} \,      
\Lambda^{\otimes (d-1)} (\dint {\bf h}) \,. \nonumber 
 \end{eqnarray} 

For the stationary STIT tessellation $ Y_{s_{d-1}}$ we consider the induced one-dimensional tessellation $Y'_{s_{d-1}} = Y_{s_{d-1}}\cap \overline{\bf h}$  as a marked point process (centers of the segments, marked with the lengths of the segments). Its intensity, that is, the mean number of segment centres per unit length on $\overline{\bf h}$ is equal to $b(u)\, s_{d-1}$. 

Denote by $\QQ_l$ the  distribution of the length of the typical segment, which is the exponential distribution with parameter $b(u)\, s_{d-1}$. Then the stationarity of the STIT tessellation and the refined Campbell theorem for marked point processes \cite[Theorem 3.5.3]{schn/weil} imply for the inner integral that
\begin{eqnarray*} 
I&:=& \int \sum_{z\in  y_{s_{d-1}}} \textstyle{
 \, 
V_j(z\cap \overline{\bf h})\, {\bf 1}\{ c( z\cap \overline{\bf h})\in [0,1]^d\}} \\
&& \qquad\qquad\qquad\times \frac{[V_1 ( z\cap \overline{\bf h}) b(u)  a({\bf s})]^n}{n!}
  e^{-V_1 ( z\cap \overline{\bf h}) b(u)  a({\bf s})}\,
 \PP_{{ Y}_{s_{d-1}}} \dint { y}_{s_{d-1}} \\ && \\
 &=& \int \sum_{z'\in  y'_{s_{d-1}}} \textstyle{
 \, 
V_j(z')\, {\bf 1}\{ c(z')\in [0,1]^d \}}   \frac{[V_1 ( z') b(u)  a({\bf s})]^n}{n!}
  e^{-V_1 ( z') b(u)  a({\bf s})}\,
 \PP_{Y'_{s_{d-1}}} \dint y'_{s_{d-1}} \\ && \\
 &=& b(u)\, s_{d-1} V_1 ([0,1]^d \cap \overline{\bf h}) 
 \int x^j\,   \frac{[x b(u)  a({\bf s})]^n}{n!}
  e^{-x b(u)  a({\bf s})}\,
 \QQ_l (\dint x) \\ && \\
 &=& b(u)\, s_{d-1} V_1 ([0,1]^d \cap \overline{\bf h}) 
 \int_0^\infty x^j\,   \frac{[x b(u)  a({\bf s})]^n}{n!}
  e^{-x b(u)  a({\bf s})}
  b(u)\, s_{d-1} e^{-b(u)\, s_{d-1}x}\,\dint x\,.
\end{eqnarray*}   
Integration yields  
\begin{eqnarray*} 
I &=& \begin{cases}  
     V_1 ([0,1]^d \cap \overline{\bf h}) \, \displaystyle{\frac{(n+1)a({\bf s})^n s_{d-1}^2}{ (a({\bf s}) +s_{d-1})^{n+2}}}  & \mbox{ if } j=1\, ,  \\  \\
    \, b(u)\, V_1 ([0,1]^d \cap \overline{\bf h})\,  
    \displaystyle{\frac{a({\bf s})^n s_{d-1}^2}{(a({\bf s}) +s_{d-1})^{n+1}}} & \mbox{ if } j=0\,.
  \end{cases}
 \end{eqnarray*}

Now we compute the inner integral on the right-hand side of \eqref{eq:rhoformel} for the special choices $B=[0,1]^d$ and $k=1$ in the same way:
\begin{eqnarray*} 
I_\rho&:=& \int \sum_{z\in  y_{s_{d-1}}} \textstyle{
 \, 
V_j(z\cap \overline{\bf h})\, {\bf 1}\{ c( z\cap \overline{\bf h})\in [0,1]^d\}}   \,
 \PP_{{ Y}_{s_{d-1}}} (\dint { y}_{s_{d-1}}) \\ && \\
 &=& \int \sum_{z'\in  y'_{s_{d-1}}} \textstyle{
 \, 
V_j(z')\, {\bf 1}\{ c(z')\in [0,1]^d \}}  \,
 \PP_{Y'_{s_{d-1}}} (\dint y'_{s_{d-1}}) \\ && \\
 &=& b(u)\, s_{d-1} V_1 ([0,1]^d \cap \overline{\bf h}) 
 \int x^j\,   
 \QQ_l (\dint x) \\ && \\
 &=& b(u)\, s_{d-1} V_1 ([0,1]^d \cap \overline{\bf h}) 
 \int_0^\infty x^j\,   
  b(u)\, s_{d-1} e^{-b(u)\, s_{d-1}x}\,\dint x 
\end{eqnarray*} 
and thus we obtain
\begin{eqnarray*} 
I_\rho &=& \begin{cases}  
     V_1 ([0,1]^d \cap \overline{\bf h}) & \mbox{ if } j=1\, ,   \\ 
    b(u)  V_1 ([0,1]^d \cap \overline{\bf h})\, s_{d-1}  
     & \mbox{ if } j=0\,.
  \end{cases}
 \end{eqnarray*} 
Combining these results leads to
\begin{eqnarray*}
 \r_{1,t}^{(j)}
&=&  \begin{cases} \displaystyle{\frac{2^{d-2} \, t^{d-1}}{(d-1)!} } \int  V_1 ([0,1]^d \cap \overline{\bf h}) \,\Lambda^{\otimes (d-1)} (\dint {\bf h}) & \mbox{ if } j=1\, ,   \\  \\
\displaystyle{\frac{2^{d-2} \, t^{d}}{d(d-2)!} }  \int   b(u)\,   V_1 ([0,1]^d)\, \Lambda^{\otimes (d-1)} (\dint {\bf h}) & \mbox{ if } j=0\,. 
 \end{cases}
\end{eqnarray*}
Finally, plugging  the inner integral $I$ and  the expression for $\r_{1,t}^{(j)}$ into \eqref{eq:punktzahl}, yields the assertion of Theorem \ref{thm:GeneralInternalVertices}.   \hfill $\Box$

\bibliographystyle{plain} 
\bibliography{literatur_nagel2}

\end{document}